\documentclass[10pt]{article}
\usepackage{lmodern}
\usepackage[T1]{fontenc}
\usepackage{amsmath}
\usepackage{filecontents}
\usepackage{amsthm}
\usepackage{amssymb}
\usepackage{mathabx} % For \bigominus
\usepackage{url}
\usepackage{latexsym}
\usepackage{titlefoot}
\usepackage[small]{titlesec}
\usepackage[small,it]{caption}
\usepackage{mathdots}
\usepackage{mathtools}
\usepackage{bm} % Bold Math package, for table headers
\usepackage{multirow}

\makeatletter

\let\c@table\c@figure
\makeatother 

\usepackage{tikz}
\usetikzlibrary{shapes,snakes,plotmarks}
\tikzstyle{ball}=[circle, draw, inner sep=1pt, color=black]
\tikzstyle{open_y}=[ball,fill=white]%, label=right:{$y$}]
\tikzstyle{open}=[ball,fill=white]
\tikzstyle{open2}=[ball,fill=white,minimum size=.15cm]
\tikzstyle{closed}=[ball,fill=black]%, label=right:{$x$}]
\tikzstyle{zball}=[ball,rectangle,minimum size=.1cm,fill=white]%, label=right:{$z$}]

\usepackage[labelformat=simple,skip=-2pt]{subcaption}

\usepackage{color}
\definecolor{lightgray}{rgb}{0.8, 0.8, 0.8}
\definecolor{darkgray}{rgb}{0.7, 0.7, 0.7}
\definecolor{darkblue}{rgb}{0, 0, .4}

\usepackage[bookmarks]{hyperref}
\hypersetup{
        colorlinks=true,
        linkcolor=darkblue,
        anchorcolor=darkblue,
        citecolor=darkblue,
        urlcolor=darkblue,
        pdfpagemode=UseThumbs,
        pdftitle={Pattern-avoiding involutions: exact and asymptotic enumeration},
        pdfsubject={Combinatorics},
        pdfauthor={Bona, Homberger, Pantone, and Vatter},
}

\newcounter{todocounter}

\theoremstyle{plain}
\newtheorem{theorem}{Theorem}[section]
\newtheorem{proposition}[theorem]{Proposition}

\newtheorem{fact}[theorem]{Fact}

\makeatletter
\newfont{\footsc}{cmcsc10 at 8truept}
\newfont{\footbf}{cmbx10 at 8truept}
\newfont{\footrm}{cmr10 at 10truept}
\renewenvironment{abstract}%
                {
                  \begin{list}{}%
                     {\setlength{\rightmargin}{1in}%
                      \setlength{\leftmargin}{1in}}%
                   \item[]\ignorespaces\begin{small}}%
                 {\end{small}\unskip\end{list}}

\newcommand{\Av}{\operatorname{Av}}

\newcommand{\ds}{\displaystyle}
\newcommand{\pa}[1]{\left({#1}\right)}

\newcommand{\f}[2]{\ds\frac{{#1}}{{#2}}}

\newcommand{\C}{\mathcal{C}}

\newcommand{\gr}{\mathrm{gr}}
\newcommand{\lgr}{\underline{\mathrm{gr}}}
\newcommand{\ugr}{\overline{\mathrm{gr}}}

\renewcommand{\a}{\mathsf{a}}
\renewcommand{\b}{\mathsf{b}}
\renewcommand{\c}{\mathsf{c}}
\renewcommand{\d}{\mathsf{d}}

\newcommand{\OEISlink}[1]{\href{http://oeis.org/#1}{#1}}
\newcommand{\OEISref}{\href{http://oeis.org/}{OEIS}~\cite{sloane:the-on-line-enc:}}
\newcommand{\OEIS}[1]{sequence \OEISlink{#1} in the \OEISref}

\newcommand{\eval}[2][\right]{\relax\ifx#1\right\relax \left.\fi#2#1\rvert}

\newcommand{\st}{\::\:}
\newcommand{\fp}{\operatorname{fp}}
\newcommand{\lrmin}{\operatorname{lrmin}}
\newcommand{\rlmax}{\operatorname{rlmax}}

\datefoot{\today}
\amssubj{05A05, 05A15}
\newpagestyle{main}[\small]{
        \headrule
        \sethead[\usepage][][]
        {\sc Pattern-Avoiding Involutions}{}{\usepage}}

\title{\sc Pattern-Avoiding Involutions: Exact and Asymptotic Enumeration}
\author{%
Mikl\'os B\'ona, Cheyne Homberger, Jay Pantone, and Vincent Vatter\footnotemark[\value{footnote}]\footnote{Vatter's research was sponsored by the National Security Agency under Grant Number H98230-12-1-0207 and the National Science Foundation under Grant Number DMS-1301692.  The United States Government is authorized to reproduce and distribute reprints not-withstanding any copyright notation herein.}\\[-0.25ex]
\small Department of Mathematics\\[-0.5ex]
\small University of Florida\\[-0.5ex]
\small Gainesville, Florida USA\\[-1.5ex]
}

\titleformat{\section}
        {\large\sc}
        {\thesection.}{1em}{}

\date{}

\begin{document}
\maketitle

\pagestyle{main}

\begin{abstract}
We consider the enumeration of pattern-avoiding involutions, focusing in particular on sets defined by avoiding a single pattern of length $4$. As we demonstrate, the numerical data for these problems exhibits some surprising behavior. This strange behavior even provides some very unexpected data related to the number of $1324$-avoiding permutations.
\end{abstract}

\section{Introduction}

For the past twenty-five years, there has been considerable interest in the enumeration of pattern-avoiding permutations. Much less work has been devoted to pattern-avoiding \emph{involutions}, the topic of this paper. We begin with preliminary definitions.

Given permutations $\pi$ and $\sigma$, considered as sequences of positive integers (one-line notation), we say that $\pi$ \emph{contains} $\sigma$, and write $\sigma\le\pi$, if $\pi$ has a subsequence $\pi(i_1)\cdots\pi(i_k)$ of the same length as $\sigma$ which is order isomorphic to $\sigma$ (i.e.,~$\pi(i_s) < \pi(i_t)$ if and only if $\sigma(s) < \sigma(t)$ for all $1\le s,t\le k$); otherwise, we say that $\pi$ \emph{avoids} $\sigma$. For example, $\pi=391867452$ contains $\sigma=51342$, as can be seen by considering the subsequence $\pi(2)\pi(3)\pi(5)\pi(6)\pi(9)=91672$.

Containment is a partial order on permutations, and we refer to downsets of permutations as \emph{permutation classes}. Thus if $\C$ is a permutation class containing $\pi$ and $\sigma\le\pi$ then $\sigma\in\C$. Given a set $B$ of permutations, we denote by $\Av(B)$ the class of permutations defined by avoiding every permutation in $B$, i.e.,
\[
\Av(B)=\{\pi\st\text{$\pi$ avoids every $\beta\in B$}\}.
\]
Conversely, for every class $\C$ there is a unique antichain $B$ such that $\C=\Av(B)$, which is called the \emph{basis} of the class.

For any permutation class $\C$, we denote by $\C_n$ the subset of permutations in $\C$ of length $n$. The generating function (by length) of $\C$ is then
\[
\sum_{n\ge 1} |\C_n|x^n=\sum_{\pi\in\C} x^{|\pi|}.
\]
(Our generating functions do not count the empty permutation.) Two permutation classes with the same enumerations are said to be \emph{Wilf-equivalent}.

In this paper we are interested in counting pattern-avoiding \emph{involutions}. Thus adapting our notation from permutation classes we write
\[
\Av^I(B)=\{\text{involutions }\pi\st \text{$\pi$ avoids every $\beta\in B$}\},
\]
however, two import caveats should be made. The first is that $\Av^I(B)$ is \emph{not} a permutation class in general. Also, note that the choice of $B$ is not unique. We define the generating functions and the notion of Wilf-equivalence for sets sets of the form $\Av^I(B)$ as we did for permutation classes.

The case where $B$ is a singleton has received considerable attention; we call such classes \emph{principal}. Much of the early work in the area of permutation patterns concerned principal classes for short patterns $\beta$. For $|\beta|=3$, there are only two different permutation classes up to symmetry, and both are well-known to be counted by the Catalan numbers (for $\beta=123$, it can be argued that this is due to MacMahon~\cite[Volume I, Section III, Chapter V]{macmahon:combinatory-ana:}, while the $\beta=231$ case was first studied by Knuth~\cite[Section 2.2.1, Exercises 4 and 5]{knuth:the-art-of-comp:1}).

The enumeration of sets of involutions avoiding a pattern $\beta$ of length $3$ was first considered in the seminal paper of Simion and Schmidt~\cite{simion:restricted-perm:}. They showed that for $\beta\in\{123,132,213,321\}$,
	\[
	|\Av^I_n(\beta)|={n\choose \lfloor n/2\rfloor},
	\]
while for $\beta\in\{231,312\}$,
	\[
	|\Av^I_n(\beta)|=2^{n-1}.
	\]

The situation gets much more complicated when $|\beta|=4$. In this case, it follows from the work of Stankova~\cite{stankova:forbidden-subse:} and Backelin, West, and Xin~\cite{backelin:wilf-equivalenc:} that there are three Wilf-equivalence classes for permutations, represented by
\[
\Av(1342),\ \Av(1234),\ \text{and}\ \Av(1324).
\]
The class $\Av(1342)$ was first enumerated by B\'ona~\cite{bona:exact-enumerati:}, who showed that it has an algebraic generating function. (Recently, a much simpler proof has been given by Bloom and Elizalde~\cite{bloom:pattern-avoidan:}.) The class $\Av(1234)$ was first enumerated by Gessel~\cite{gessel:symmetric-funct:}, who showed that it has a $D$-finite but nonalgebraic generating function. Unlike the other two classes, $\Av(1324)$ has resisted all attempts to determine its exact enumeration.

\begin{table}
$$
\begin{array}{ccccccccc}%{|c|c|c|c|c|c|c|c|c|}
\hline
&&&&&&&&\\[-10pt]
&\bm{1324}&\bm{1234}&\bm{4231}&\bm{2431}&\bm{1342}&\bm{2341}&\bm{3421}&\bm{2413}\\[1pt]\hline
&&&&&&&&\\[-11pt]
|\Av^I_{5}(\beta)|&21&21&21&24&24&25&25&24\\[1pt]\hline
&&&&&&&&\\[-11pt]
|\Av^I_{6}(\beta)|&51&51&51&62&62&66&66&64\\[1pt]\hline
&&&&&&&&\\[-11pt]
|\Av^I_{7}(\beta)|&126&127&128&154&156&170&173&166\\[1pt]\hline
&&&&&&&&\\[-11pt]
|\Av^I_{8}(\beta)|&321&323&327&396&406&441&460&456\\[1pt]\hline
&&&&&&&&\\[-11pt]
|\Av^I_{9}(\beta)|&820&835&858&992&1040&1124&1218&1234\\[1pt]\hline
&&&&&&&&\\[-11pt]
|\Av^I_{10}(\beta)|&2160&2188&2272&2536&2714&2870&3240&3454\\[1pt]\hline
&&&&&&&&\\[-11pt]
|\Av^I_{11}(\beta)|&5654&5798&6146&6376&7012&7273&8602&9600\\[1pt]\hline
\end{array}
$$
\caption{The enumerations of involutions avoiding a pattern $\beta$ of length $4$ for $n=5$, $\dots$, $11$, as presented by Jaggard~\cite{jaggard:prefix-exchangi:}.}
\label{table-enum-1}
\end{table}

Extending the work of Guibert~\cite{guibert:combinatoire-de:} and Guibert, Pergola, and Pinzani~\cite{guibert:vexillary-invol:}, Jaggard~\cite{jaggard:prefix-exchangi:} completed the classification of Wilf-equivalence classes of sets of involutions avoiding a pattern of length $4$. Of these eight Wilf-classes, only two enumerations have been computed so far. Gessel~\cite{gessel:symmetric-funct:} showed that $\Av^I(1234)$ is counted by the Motzkin numbers, while Brignall, Huczynska, and Vatter~\cite{brignall:simple-permutat:} counted $\Av^I(2413)$. It should be remarked that presenting such sets as ``principal'' is a bit disingenuous; any involution which avoids $2413$ must also avoid $2413^{-1}=3142$, so $\Av^I(2413)=\Av^I(2413,3142)$. Still, we keep with the established tradition and present the (or more accurately, one of the) shortest possible bases for such sets.

Jaggard concluded his paper by presenting the first few terms of the enumerations of these Wilf-equivalence classes, which we show in Table~\ref{table-enum-1}. The order of the columns of this table is determined by the number of permutations of length $11$ which avoid each pattern. Our interest in the topic of pattern-avoiding involutions was first piqued when we noticed that this ordering \emph{is incredibly misleading} (through no fault of Jaggard's).

In the next section we discuss the asymptotic enumeration of pattern-avoiding permutations and involutions and show why the order of the columns of Table~\ref{table-enum-1} must be incorrect (for large values of $n$). Then in Section~\ref{sec-simples} we give an overview of the method we use to provide two new enumerations --- those of $\Av^I(1342)$ and $\Av^I(2341)$. Section~\ref{sec-123} contains some rather technical calculations which we will need to prove these results while the results themselves are proved in Sections~\ref{sec-1342} and \ref{sec-2341}.

\section{Growth Rates and the Deceptiveness of Table~\ref{table-enum-1}}\label{section:grs}

To explain why the ordering of the columns of Table~\ref{table-enum-1} must be incorrect, we want to look at the asymptotic, rather than exact, enumeration of such sets. The Marcus-Tardos Theorem~\cite{marcus:excluded-permut:} (formerly the Stanley-Wilf Conjecture) states that all permutation classes other than the class of all permutations have at most exponential growth, i.e., for every class $\C$ with a nonempty basis, there is a constant $K\ge 0$ so that $\C$ contains at most $K^n$ permutations of length $n$ for all $n$. Thus every nondegenerate permutation class $\C$ has finite {\it upper\/} and {\it lower growth rates\/} defined, respectively, by
\[
	\ugr(\C) = \limsup_{n\rightarrow\infty}\sqrt[n]{|\C_n|}
		\qquad \text{ and } \qquad
	\lgr(\C) = \liminf_{n\rightarrow\infty}\sqrt[n]{|\C_n|}.
\]
It is conjectured that every permutation class has a {\it proper growth rate\/}, and when we are dealing with a class for which $\ugr(\C)=\lgr(\C)$, we denote this quantity by $\gr(\C)$. Clearly, sets of the form $\Av^I(B)$ have analogous upper and lower growth rates, which we denote similarly, and if these two quantities agree, we call that quantity the proper growth rate of the set.

Arratia~\cite{arratia:on-the-stanley-:} showed that principal classes always have proper growth rates, which are in this case sometimes called \emph{Stanley-Wilf limits}. We briefly recount his proof now, partly because we need to use this machinery in our proofs. The \emph{direct sum} (or just \emph{sum} for short) of the permutations $\sigma$ of length $m$ and $\tau$ of length $n$ is the permutation $\sigma\oplus\tau$ defined by
	\[
	(\sigma\oplus\tau)(i) =
	\left\{
	\begin{array}{ll}
	\sigma(i)&\mbox{for $1\le i\le m$,}\\
	\tau(i-m)+m&\mbox{for $m+1\le i\le m+n$.}
	\end{array}
	\right.
	\]
There is also an obvious symmetry of the sum operation called \emph{skew sum}; both of these operations are shown in Figure~\ref{fig-sums}. The permutation $\pi$ is said to be \emph{sum (resp., skew) indecomposable} if it cannot be expressed as a sum (resp., skew sum) of two proper subpermutations.

\begin{figure}
\begin{center}
	$\pi\oplus\sigma=$
	\begin{tikzpicture}[scale=0.5, baseline=(current bounding box.center)]
		\draw (0,0) rectangle (1,1);
		\draw (1,1) rectangle (2,2);
		\node at (0.5,0.5) {$\pi$};
		\node at (1.5,1.5) {$\sigma$};
	\end{tikzpicture}
\quad\quad\quad\quad
	$\pi\ominus\sigma=$
	\begin{tikzpicture}[scale=0.5, baseline=(current bounding box.center)]
		\draw (0,1) rectangle (1,2);
		\draw (1,0) rectangle (2,1);
		\node at (0.5,1.5) {$\pi$};
		\node at (1.5,0.5) {$\sigma$};
	\end{tikzpicture}
\end{center}
\caption{The sum and skew sum operations.}
\label{fig-sums}
\end{figure}

The permutation class $\C$ is said to be \emph{sum closed} if $\sigma\oplus\tau\in\C$ for every $\sigma,\tau\in\C$ (the term \emph{skew closed} is defined analogously). It is not hard to see that the class $\Av(B)$ is sum (resp., skew) closed if and only if every permutation $\beta\in B$ is sum (resp., skew) indecomposable. Note that every permutation is either sum or skew indecomposable. Therefore, a principal class must be either sum or skew closed. 

It is then easy to see that principal classes have proper growth rates. Suppose that $\C$ is a principal class. By symmetry we may assume that $\C$ is sum closed. Therefore the sum operation defines an injection
	\[
	\oplus\st\C_m\times\C_n\rightarrow\C_{m+n}.
	\]
Therefore the sequence $\{|\C_n|\}$ is supermultiplicative, i.e., $|\C_{m+n}|\ge|\C_m||\C_n|$. It then follows from Fekete's Lemma that the growth rate of $\C$ exists (though Fekete's Lemma allows the limit to be infinite, this possibility is ruled out by the Marcus-Tardos Theorem).

For counting involutions, we cannot make such a strong claim. Indeed, $1\ominus 12=312$ is not an involution, so no nontrivial sets of the form $\Av^I(B)$ are skew closed. However, it is still true that $\sigma\oplus\tau$ is an involution whenever both $\sigma$ and $\tau$ are, so we can get roughly half of Arratia's result:

\begin{proposition}
\label{prop-gr-involution}
If every permutation in $B$ is sum indecomposable then $\Av^I(B)$ has a proper growth rate.
\end{proposition}

For $|\beta|=4$ there are three possible growth rates of principal classes of the form $\Av(\beta)$. Regev~\cite{regev:asymptotic-valu:} showed that $\gr(\Av(1234))=9$, while B\'ona's work~\cite{bona:exact-enumerati:} shows that $\gr(\Av(1342))=8$. The final value, that of $\gr(\Av(1324))$, is currently unknown, although we have bounds in both directions. For the upper-bound, B\'ona~\cite{bona:a-new-record-fo:,bona:a-new-upper-bou:} has extended an argument of Claesson, Jel{\'{\i}}nek, and Steingr{\'{\i}}msson~\cite{claesson:upper-bounds-fo:} to show that $\gr(\Av(1324))\le (2+\sqrt{3})^2\approx 13.74$. The best current lower bound on $\gr(\Av(1324))$ is $9.81$, due to Bevan~\cite{bevan:a-large-set-of-:}, while Conway and Guttman~\cite{conway:on-the-growth-r:} have estimated that $\gr(\Av(1324))\approx 11.60$.

Next we provide a relation between $\beta$-avoiding permutations and $\beta$-avoiding involutions in the case where $\beta$ is a skew indecomposable involution (such as $\beta=1324$, which is the case we want it for).

\begin{proposition}
\label{prop-skew-indecomp-lower-bound}
For every skew indecomposable involution $\beta$, we have
$$
\ugr(\Av^I(\beta))\ge\sqrt{\gr(\Av(\beta))}.
$$
\end{proposition}
\begin{proof}
Suppose that $\beta$ is a skew indecomposable involution and take a permutation $\pi\in\Av_n(\beta)$. Because $\beta$ is an involution, $\pi^{-1}$ must also avoid $\beta$. Moreover, because $\beta$ is skew indecomposable, $\pi\ominus\pi^{-1}$ will avoid $\beta$. Note that $\pi\ominus\pi^{-1}$ is an involution for every permutation $\pi$, so under our hypotheses the mapping $\pi\mapsto\pi\ominus\pi^{-1}$ defines an injection from $\Av_n(\beta)$ to $\Av_{2n}^I(\beta)$. It follows that
$$
\ugr(\Av^I(\beta))
\ge
\limsup_{n\rightarrow\infty} \sqrt[2n]{|\Av_{2n}^I(\beta)|}
\ge
\limsup_{n\rightarrow\infty} \sqrt[2n]{|\Av_n(\beta)|}
=
\sqrt{\gr(\Av(\beta))},
$$
as desired.
\end{proof}

\begin{table}
\begin{footnotesize}
$$
\begin{array}{ccccccccc}%{|c|c|c|c|c|c|c|c|c|}
\hline
&&&&&&&&\\[-8pt]
&\bm{2431}&\bm{2341}&\bm{1342}&\bm{1234}&\bm{1324}&\bm{3421}&\bm{4231}&\bm{2413}\\[1pt]\hline
&&&&&&&&\\[-9pt]
|\Av^I_{12}(\beta)|&
16238&18477&18322&15511&15272&22878&16716&27246
\\[1pt]\hline
&&&&&&&&\\[-9pt]
|\Av^I_{13}(\beta)|&
40914&46825&47560&41835&40758&60794&46246&77132
\\[1pt]\hline
&&&&&&&&\\[-9pt]
|\Av^I_{14}(\beta)|&
103954&118917&124358&113634&112280&161668&128414&221336
\\[1pt]\hline
&&&&&&&&\\[-9pt]
|\Av^I_{15}(\beta)|&
262298&301734&323708&310572&304471&429752&361493&635078
\\[1pt]\hline
&&&&&&&&\\[-9pt]
|\Av^I_{16}(\beta)|&
665478&766525&846766&853467&852164&1142758&1020506&1839000
\\[1pt]\hline
&&&&&&&&\\[-9pt]
|\Av^I_{17}(\beta)|&
1680726&1946293&2208032&2356779&2341980&3038173&2913060&5331274
\\[1pt]\hline
&&&&&&&&\\[-9pt]
|\Av^I_{18}(\beta)|&
4260262&4944614&5777330&6536382&6640755&8078606&8335405&15555586
\\[1pt]\hline
&&&&&&&&\\[-9pt]
|\Av^I_{19}(\beta)|&
10766470&12557685&15082372&18199284&18460066&21479469&24067930&45465412
\\[1pt]\hline
&&&&&&&&\\[-9pt]
|\Av^I_{20}(\beta)|&
27274444&31900554&39469786&50852019&52915999&57113888&69646035&133517130
\\[1pt]\hline
&&&&&&&&\\[-9pt]
\text{growth}&
\multirow{2}{*}{?}&\approx 2.54&\approx 2.62&3&> 3.13,\;< 4.84&\multirow{2}{*}{?}&\multirow{2}{*}{?}&\approx 3.15
\\[1pt]
&&&&&&&&\\[-9pt]
\text{rate}&
&\text{Section~\ref{sec-2341}}&\text{Section~\ref{sec-1342}}&\text{Regev~\cite{regev:asymptotic-valu:}}&\text{Sections~\ref{section:grs} \&~\ref{section:1324-revisited}}&&&\text{BHV~\cite{brignall:simple-permutat:}}
\\[1pt]\hline
&&&&&&&&\\[-9pt]
\text{OEIS}&\OEISlink{A230551}&\OEISlink{A230552}&\OEISlink{A230553}&\OEISlink{A001006}&\OEISlink{A230554}&\OEISlink{A230555}&\OEISlink{A230556}&\OEISlink{A121704}\\[1pt]\hline
\end{array}
$$
\end{footnotesize}
\caption{The enumerations of involutions avoiding a pattern of length $4$ for $n=12$, $\dots$, $20$, with columns sorted according to the number of involutions of length $20$ avoiding the given pattern. Note that the bounds for $1324$ are actually bounds of the upper growth rate.}
\label{table-enum-2}
\end{table}

Bevan's bound on $\gr(\Av(1324))$ and Proposition~\ref{prop-skew-indecomp-lower-bound} therefore imply that
$$
\ugr(\Av^I(1324))>3.13,
$$
and thus the number of $1324$-avoiding involutions must overtake the number of $1234$-avoiding involutions at some point (these have the growth rate $3$ by Regev~\cite{regev:asymptotic-valu:}). Moreover, the number of $1324$-avoiding involutions should overtake the number of $2413$-avoiding involutions (which have a growth rate of approximately $3.15$), unless $\gr(\Av(1324))<9.9$, which seems incredibly unlikely.

We conclude this section by updating Jaggard's table to include data up to $n=20$ in Table~\ref{table-enum-2}, which was computed using Albert's PermLab package~\cite{PermLab1.0}. This data shows that the number of $1324$-avoiding involutions first overtakes the number of $1234$-avoiding involutions at $n=18$, and does not overtake the number of $2413$-avoiding involutions for $n\le 20$. Thus by our remarks above, \emph{the ordering of the columns in Table~\ref{table-enum-2} is likely still incorrect}.

\section{Simple Permutations and Separable Involutions}
\label{sec-simples}

Our principal tool in what follows is the \emph{substitution decomposition} of permutations into intervals.  An \emph{interval} in the permutation $\pi$ is a set of contiguous indices $I=[a,b]=\{a,a+1,\dots,b\}$ such that the set of values $\pi(I)=\{\pi(i) : i\in I\}$ is also contiguous.  Given a permutation $\sigma$ of length $m$ and nonempty permutations $\alpha_1,\dots,\alpha_m$, the \emph{inflation} of $\sigma$ by $\alpha_1,\dots,\alpha_m$,  denoted $\sigma[\alpha_1,\dots,\alpha_m]$, is the permutation of length $|\alpha_1|+\cdots+|\alpha_m|$ obtained by replacing each entry $\sigma(i)$ by an interval that is order isomorphic to $\alpha_i$ in such a way that the permutation of the intervals is order isomorphic to $\sigma$.  For example,
\[
2413[1,132,321,12]=4\ 798\ 321\ 56. 
\]
We have already introduced two special inflations: $\alpha_1\oplus\alpha_2$ is the same as $12[\alpha_1,\alpha_2]$, and $\alpha_1\ominus\alpha_2$ is the same as $21[\alpha_1,\alpha_2]$.

Every permutation of length $n\ge 1$ has \emph{trivial} intervals of lengths $0$, $1$, and $n$; all other intervals are termed \emph{proper}. A permutation of length at least $2$ is called \emph{simple} if it has no proper intervals.  The shortest simple permutations are thus $12$ and $21$, there are no simple permutations of length three, and the simple permutations of length four are $2413$ and $3142$.

Simple permutations and inflations are linked by the following result.

\begin{proposition}[Albert and Atkinson~\cite{albert:simple-permutat:}]
\label{simple-decomp-unique}
Every permutation $\pi$ except $1$ is the inflation of a unique simple permutation $\sigma$.  Moreover, if $\pi=\sigma[\alpha_1,\dots,\alpha_m]$ for a simple permutation $\sigma$ of length $m\ge 4$, then each interval $\alpha_i$ is unique.  If $\pi$ is an inflation of $12$ (i.e., is sum decomposable), then there is a unique sum indecomposable $\alpha_1$  such that $\pi=\alpha_1\oplus\alpha_2$.  The same holds, mutatis mutandis, with $12$ replaced by $21$ and sum replaced by skew.
\end{proposition}

To give an easy example of using the substitution decomposition to count a permutation class, we apply it to the class $\Av(2413,3142)$, known also as the \emph{separable permutations} (this enumeration was first performed by Shapiro and Stephens~\cite{shapiro:bootstrap-perco:}). It is well-known that every simple permutation of length at least four contains either $2413$ or $3142$, so the only simple permutations in this class are $12$ and $21$, i.e., every nontrivial separable permutation is either a sum or a skew sum. Let us denote by $f$ the generating function for the separable permutations, $f_\oplus$ the generating function for sum decomposable separable permutations, and $f_\ominus$ the generating function for skew decomposable separable permutations. Quite trivially, we see that
$$
f=x+f_\oplus+f_\ominus.
$$
(If this class contained more simple permutations, the equation above would also include terms counting their inflations.) By Proposition~\ref{simple-decomp-unique}, we can write every sum decomposable permutation uniquely in the form $\alpha_1\oplus\alpha_2$ where $\alpha_1$ is sum indecomposable and $\alpha_2$ is arbitrary. Since the generating function for the sum indecomposable separable permutations is $f-f_\oplus$ and the class of separable permutations is closed under sums, we have $f_\oplus=(f-f_\oplus)f$, and thus it follows that $f_\oplus={f^2}/(1+f)$. By symmetry, $f_\ominus=f_\oplus$, and thus $f=x+2f^2/(1+f)$. Solving this equation shows that the separable permutations are indeed counted by the (large) Schr\"oder numbers.

The substitution decomposition has proved to be a powerful tool for describing the structure of permutation classes. However, it seems to have been used to count pattern-avoiding involutions only once, when Brignall, Huczynska, and Vatter~\cite{brignall:simple-permutat:} enumerated the separable involutions. We first review the general principles and then (re)apply them to this case.

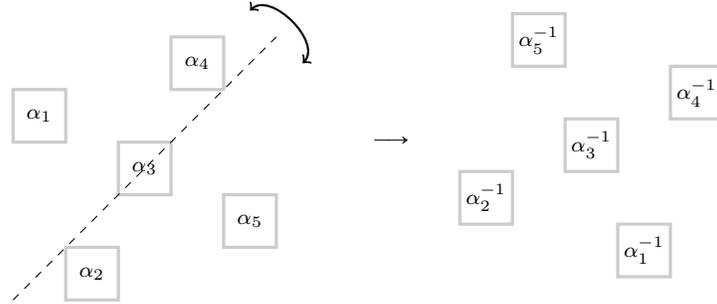
\begin{figure}
	\begin{footnotesize}
	$$
		\begin{array}{ccc}

		\begin{tikzpicture}[scale=.7,baseline=(current bounding box.center)]
		\foreach \x/\y in {1/4, 2/1, 3/3, 4/5, 5/2}{
			\draw[very thick, color=lightgray] (\x,\y) rectangle (\x+1,\y+1);
			\node at (\x+.5,\y+.5) {$\alpha_{\x}$};
		}
		\draw[dashed] (1,1)--(6,6);
		\draw[thick, <->] (5.5,6.5) to[out=45, in=45] (6.5,5.5);
		\end{tikzpicture}

		&
		\quad\longrightarrow\quad
		&
		
		\begin{tikzpicture}[scale=.7,baseline=(current bounding box.center)]
		\foreach \x/\y in {1/4, 2/1, 3/3, 4/5, 5/2}{
			\draw[very thick, color=lightgray] (\y,\x) rectangle (\y+1,\x+1);
			\node at (\y+.5,\x+.5) {$\alpha_{\x}^{-1}$};
		}
		\end{tikzpicture}
	
		\end{array}
	$$
	\end{footnotesize}
	\caption{The inverse of $41352[\alpha_1,\alpha_2,\alpha_3,\alpha_4,\alpha_5]$ is $25314[\alpha_2^{-1},\alpha_5^{-1},\alpha_3^{-1},\alpha_1^{-1},\alpha_4^{-1}]$.}
	\label{figure:inverse-inflation}
\end{figure}

We begin by considering the effect of inversion on the substitution decomposition. As illustrated in Figure~\ref{figure:inverse-inflation}, we have
	\[
	\left(\sigma[\alpha_1,\dots,\alpha_m]\right)^{-1}
	=
	\sigma^{-1}[\alpha_{\sigma^{-1}(1)}^{-1},\dots,\alpha_{\sigma^{-1}(m)}^{-1}].
	\]
If $|\sigma|\ge 4$, then the uniqueness conditions in Proposition~\ref{simple-decomp-unique} show that $\sigma[\alpha_1,\dots,\alpha_m]$ is an involution if and only if $\sigma$ is an involution and $\alpha_i=\alpha_{\sigma^{-1}(i)}^{-1}=\alpha_{\sigma(i)}^{-1}$ for all $1 \leq i \leq m$. This rule also applies to sum decomposable permutations; $\alpha_1\oplus\alpha_2$ is an involution if and only if both $\alpha_1$ and $\alpha_2$ are. We collect these observations below.

\begin{proposition}[Brignall, Huczynska, and Vatter~\cite{brignall:simple-permutat:}]
\label{involution-decomp-1}
Let $\sigma\neq 21$ be a simple permutation. Then $\pi=\sigma[\alpha_1,\dots,\alpha_m]$ is an involution if and only if $\sigma$ is an involution and $\alpha_i=\alpha_{\sigma^{-1}(i)}^{-1}=\alpha_{\sigma(i)}^{-1}$ for all $1 \leq i \leq m$.
\end{proposition}

The skew decomposable involutions require a bit more care. If $\alpha_1$ and $\alpha_2$ are both skew indecomposable, then $\alpha_1\ominus\alpha_2$ is an involution if and only if $\alpha_1=\alpha_2^{-1}$. Otherwise (in the case where we have more than two skew components) we decompose these permutations as $\alpha_1\ominus\alpha_2\ominus\alpha_3$. The characterization is below.

\begin{proposition}[Brignall, Huczynska, and Vatter~\cite{brignall:simple-permutat:}]
\label{involution-decomp-2}
The skew decomposable involutions are precisely those of the form
\begin{itemize}
\item $21[\alpha_1,\alpha_2]$ for skew indecomposable $\alpha_1$ and $\alpha_2$ with $\alpha_1=\alpha_2^{-1}$ and
\item $321[\alpha_1,\alpha_2,\alpha_3]$, where $\alpha_1$ and $\alpha_3$ are skew indecomposable, $\alpha_1=\alpha_3^{-1}$, and $\alpha_2$ is an involution.
\end{itemize}
\end{proposition}

To provide a gentle introduction to the techniques used in this paper, we now rederive the enumeration of the separable involutions from \cite{brignall:simple-permutat:}. Note that in Tables~\ref{table-enum-1} and \ref{table-enum-2}, these are listed as the $2413$-avoiding involutions, because if an involution avoids $2413$ then it must also avoid $2413^{-1}=3142$. We retain the definitions of $f$, $f_\oplus$, and $f_\ominus$ from above, and additionally let $g$ denote the generating function for $\Av^I(2413)$, $g_\oplus$ the generating function for the sum decomposable permutations in $\Av^I(2413)$, and $g_\ominus$ the generating function for the skew decomposable permutations in $\Av^I(2413)$. From Propositions~\ref{simple-decomp-unique} and \ref{involution-decomp-1} we see that as in the non-involution case, $g_\oplus=(g-g_\oplus)g$, so $g_\oplus=g^2/(1+g)$.

Next we count skew decomposable permutations. By Proposition~\ref{involution-decomp-2}, the involutions in $\Av^I(2413)$ of the form $\alpha\ominus\alpha^{-1}$ for $\alpha$ skew indecomposable are counted by
	\[
	f(x^2)-f_\ominus(x^2).
	\]
The skew decomposable inflations of $321$ in $\Av^I(2413)$ are counted by
	\[
	\left(f(x^2)-f_\ominus(x^2)\right)\cdot g.
	\]
Accounting for the trivial permutation $1$, we have the equation
	\[g = x + \frac{g^2}{1+g} + \left(f(x^2)-f_\ominus(x^2)\right)\left(1+g\right).\]
Rearranging terms gives
	\[g^2\left(f(x^2)-f_\ominus(x^2)\right) + g\left(x - 1 + 2\left(f(x^2)-f_\ominus(x^2)\right)\right) + \left(x + f(x^2)-f_\ominus(x^2)\right) = 0.\]
Finally, solving for $g$ yields the desired generating function,
	\[g = \frac{1 - 3x + x^2 + x^3 + r(1+x) - \sqrt{q}}{2(1 - r - x^2)}\]
where
	\[r = \sqrt{1-6x^2+x^4}\]
and
	\[q = -6 - 20x + 38x^2 + 24x^3 - 18x^4 - 4x^5 + 2x^6 + r(10 + 12x - 12x^2 - 4x^3 + 2x^4).\]
The growth rate of this generating function is the reciprocal of the singularity closest to the origin, and is therefore
	\[\f{1}{\sqrt{3}-\sqrt{2}} = \sqrt{2} + \sqrt{3} \approx 3.15.\]

%
%which, along with our previous computations, allows us to compute that the minimal polynomial of $g$ is
%	\[
%	x^2g^4 + (x^3+3x^2+x-1)g^3 + (3x^3+6x^2-x)g^2  + (3x^3+7x^2-x-1)g +x^3+3x^2+x=0.
%	%x+3*x^2+x^3 + (-1-x+7*x^2+3*x^3)*g + (-x+6*x^2+3*x^3)*g^2 + (-1+x+3*x^2+x^3)*g^3 + x^2*g^4=0.
%	\]
%%Setting
%%	\begin{align*}
%%		a_4 &= x^2,\\
%%		a_3 &= x^3+3x^2+x-1,\\
%%		a_2 &= 3x^3 + 6x^2 - x,\\
%%		a_1 &= 3x^3 + 7x^2 - x- 1,\\
%%		a_0 &= x^3 + 3x^2 + x,
%%	\end{align*}
%The above minimal polynomial has the form
%	\[a_4g^4 + a_3g^3 + a_2g^2 + a_1g + a_0 = 0.\]
%Using the techniques in Flajolet and Sedgewick~\cite[Section VII.7 and Appendix B.1]{flajolet:analytic-combin:}, we find that the discriminant of this minimal polynomial with main variable $g$ and parameter $x$ is
%	\begin{align*}
%		D_g(x)
%		&=
%		\operatorname{det}\br{\begin{array}{ccccccc}
%			a_4 & a_3 & a_2 & a_1 & a_0 & 0 & 0\\
%			0 & a_4 & a_3 & a_2 & a_1 & a_0 & 0\\
%			0 & 0 & a_4 & a_3 & a_2 & a_1 & a_0\\
%			4a_4 & 3a_3 & 2a_2 & a_1 & 0 & 0 & 0\\
%			0 & 4a_4 & 3a_3 & 2a_2 & a_1 & 0 & 0\\
%			0 & 0 & 4a_4 & 3a_3 & 2a_2 & a_1 & 0\\
%			0 & 0 & 0 & 4a_4 & 3a_3 & 2a_2 & a_1
%		\end{array}}\\[10pt]
%		&=
%		-4x^2(x^2-2x-1)^2(x^2+2x-1)^2(x^4-10x^2+1).
%	\end{align*}
%	
%	
%Since the smallest positive root of $D_g(x)$ is $\sqrt{3}-\sqrt{2}$, the growth rate of $g$ is
%	\[\f{1}{\sqrt{3}-\sqrt{2}} = \sqrt{2} + \sqrt{3} \approx 3.15.\]
	
%
%
%
%
%
%

\section{Simple Involutions Avoiding $123$}
\label{sec-123}

In both of the sets we enumerate, (nearly all of) the simple involutions avoid $123$, and in order to count the sets we are interested in, we need the enumeration of the simple $123$-avoiding involutions by their number of fixed points, $\fp(\sigma)$, number of left-to-right minima, $\lrmin(\sigma)$, and number of right-to-left maxima, $\rlmax(\sigma)$. Since every $123$-avoiding permutation can be expressed as the union of two decreasing sequences, every entry of such a permutation is either a left-to-right minimum, a right-to-left maximum, or both. Moreover, no entry of a \emph{simple} $123$-avoiding permutation of length at least four can be both a left-to-right minimum and a right-to-left maximum (because then the permutation would be skew decomposable). Thus the generating functions we are interested in are
	\[
	\widehat{s}^{(i)}(u,v)
	=
	\sum_{\mathclap{\substack{\textrm{simple $\sigma\in\Av^I(123)$}\\ \text{with $\fp(\sigma)=i$}}}} \; u^{\lrmin(\sigma)} v^{\rlmax(\sigma)}.
	\]
Note that $\widehat{s}^{(i)}=0$ for $i\ge 3$, because the fixed points of a permutation form an increasing subsequence.

The \emph{staircase decomposition} was introduced by Albert, Atkinson, Brignall, Ru\v{s}kuc, Smith, and West~\cite{albert:growth-rates-fo:} as part of the study of subclasses of $\Av(321)$. This decomposition was later used by Albert and Vatter~\cite{albert:generating-and-:} to explicitly enumerate the simple permutations of $\Av(321)$. As $\Av(123)$ is a symmetry of $\Av(321)$, we follow the same approach, using much of the same terminology. Before moving on to involutions, we first give a brief summary of the techniques used in~\cite{albert:generating-and-:} by mirroring their methods to enumerate the simple permutations in $\Av(123)$. 

Every simple permutation in the class $\Av(123)$ can be uniquely written as the union of two decreasing sequences. We can further partition the entries of such a permutation into the cells of a ``staircase decomposition'', whose precise definition we opt to omit in favor of an illustration, namely Figure~\ref{figure:staircase-1}.
We will however carefully define one particular type of staircase decomposition, which we call the \emph{greedy} gridding. The greedy gridding ensures that each permutation $\sigma \in \Av(123)$ can only be partitioned in exactly one way.
%For a given permutation $\sigma \in \Av(123)$, there may be several different ways of partitioning the entries into cells;
%to ensure uniqueness, we define a unique partition which we call the \emph{greedy} gridding.
To find the greedy gridding of $\sigma \in \Av(123)$, take the first cell to consist of the longest decreasing prefix of $\sigma$. Then, each new eastward cell contains all entries whose value is greater than any previously included entry, and each new southward cell contains all entries whose index is less than any previously included entry.

\begin{figure}
	\begin{center}
		\begin{tikzpicture}[scale=.25, yscale=-1]
			% draw the outer boxes, using a loop
			\foreach \x/\y in {0/0, 5/0, 5/5, 10/5}{
				\draw[very thick, color=lightgray] (\x, \y) rectangle (\x + 5, \y + 5);
			}
			% closed dots
			\node[closed] at (2,2) {};
			\node[closed] at (4,4) {};
			\node[closed] at (5.6,.6) {};
			\node[closed] at (5.9,5.9) {};
			\node[closed] at (6.2,1.2) {};
			\node[closed] at (7.6,7.6) {};
			\node[closed] at (8.4,3.4) {};
			\node[closed] at (10.5,5.5) {};
			\node[closed] at (12,7) {};
		\end{tikzpicture}
	\end{center}	
	\caption{The staircase decomposition for the permutation $759381642$.}
	\label{figure:staircase-1}
\end{figure}
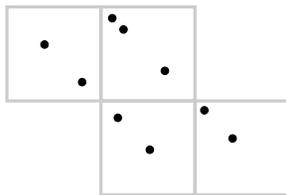

The iterative construction of a simple $123$-avoiding permutation can then be described using a sequence of hollow dots and filled dots. A hollow dot represents a space where we must add a nonempty decreasing sequence of entries in the next step, while the filled dots represent the entries themselves once they have been placed. Hollow dots can only exist in the outermost nonempty cell of each step in the recurrence. To preserve simplicity, whenever a hollow dot is filled by two or more entries in the next step of the recurrence, each neighboring pair of these entries must be split by a hollow dot. Figure~\ref{figure:staircase-2} demonstrates the steps of the recurrence which would build the permutation shown in Figure~\ref{figure:staircase-1}.

\begin{figure}
	\begin{center}
		\begin{tikzpicture}[scale=.2, yscale=-1]

			\foreach \x/\y in {0/0}{
				\draw[very thick, color=lightgray] (\x, \y) rectangle (\x + 5, \y + 5);
			}
			\node[open2] at (2.5,2.5) {};
			
		\node at (7.5,2.5) {$\longrightarrow$};

		\begin{scope}[shift={(10,0)}]
			\foreach \x/\y in {0/0, 5/0}{
				\draw[very thick, color=lightgray] (\x, \y) rectangle (\x + 5, \y + 5);
			}
			\draw (3,3)--(8,3);
			\draw (1,1)--(6,1);
			\node[closed] at (2,2) {};
			\node[closed] at (4,4) {};
			\node[open2] at (6,1) {};
			\node[open2] at (8,3) {};
		\end{scope}
		
		\node at (22.5,2.5) {$\longrightarrow$};
		
		\begin{scope}[shift={(25,0)}]
			\foreach \x/\y in {0/0, 5/0, 5/5}{
				\draw[very thick, color=lightgray] (\x, \y) rectangle (\x + 5, \y + 5);
			}
			\draw (5.9,1)--(5.9,5.9);
			\draw (7.6,2.6)--(7.6,7.6);
			\node[closed] at (2,2) {};
			\node[closed] at (4,4) {};
			\node[closed] at (5.6,.6) {};
			\node[open2] at (5.9,5.9) {};
			\node[closed] at (6.4,1.4) {};
			\node[open2] at (7.6,7.6) {};
			\node[closed] at (8.4,3.4) {};
		\end{scope}
		
		\node at (37.5,2.5) {$\longrightarrow$};

		\begin{scope}[shift={(40,0)}]
			\foreach \x/\y in {0/0, 5/0, 5/5, 10/5}{
				\draw[very thick, color=lightgray] (\x, \y) rectangle (\x + 5, \y + 5);
			}
			\draw (7,7)--(12,7);
			\draw (5.7,5.7)--(10.7,5.7);
			\node[closed] at (2,2) {};
			\node[closed] at (4,4) {};
			\node[closed] at (5.6,.6) {};
			\node[closed] at (6.1,6.1) {};
			\node[closed] at (6.4,1.4) {};
			\node[closed] at (7.6,7.6) {};
			\node[closed] at (8.4,3.4) {};
			\node[open2] at (10.7,5.7) {};
			\node[open2] at (12,7) {};
		\end{scope}
		
		\node at (57.5,2.5) {$\longrightarrow$};

		\begin{scope}[shift={(60,0)}]
			\foreach \x/\y in {0/0, 5/0, 5/5, 10/5}{
				\draw[very thick, color=lightgray] (\x, \y) rectangle (\x + 5, \y + 5);
			}
			\node[closed] at (2,2) {};
			\node[closed] at (4,4) {};
			\node[closed] at (5.6,.6) {};
			\node[closed] at (6,6) {};
			\node[closed] at (6.2,1.2) {};
			\node[closed] at (7.6,7.6) {};
			\node[closed] at (8.4,3.4) {};
			\node[closed] at (10.5,5.5) {};
			\node[closed] at (12,7) {};
		\end{scope}

		\end{tikzpicture}
	\end{center}	
	\caption{The evolution of the permutation $759381642$ by our recurrence.}
	\label{figure:staircase-2}
\end{figure}
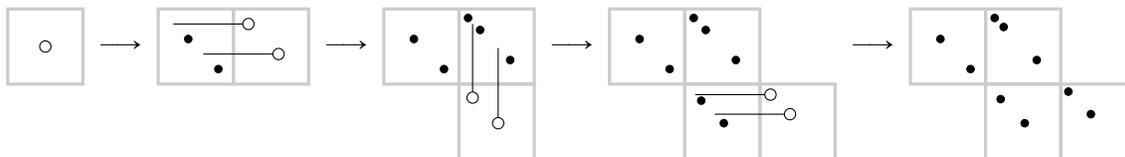

Formally, let $s_i(x,y)$ count the possible configurations in the $i$th stage of the recurrence, where hollow dots are counted by $y$ and filled dots are counted by $x$. We follow the exposition in~\cite{albert:generating-and-:} by first giving a ``mostly correct'' derivation, and then correcting two small errors to obtain the correct result.

The first stage of the recurrence is obviously counted by 
	\[s_1(x,y) = y.\]
In the second step, we may inflate this entry $y$ by some decreasing sequence of entries. Each pair of these entries must be separated by a hollow dot in the next cell, and we have the option of having a hollow dot above the first entry. Therefore, the second step can be counted by the generating function
	\[x(1+y) + x^2(y+y^2) + x^3(y^2 + y^3) + \cdots = \f{x(1+y)}{1-xy},\]
so that
	\[s_2(x,y) = s_1\pa{x, \f{x(1+y)}{1-xy}}.\]
Each subsequent step has the same recurrence, so that
	\[s_{n+1}(x,y) = s_n\pa{x, \f{x(1+y)}{1-xy}}.\]
The desired generating function $s(x)$ is found by taking the limit as $n \to \infty$ (a procedure which is explained in more detail by Albert and Vatter~\cite{albert:generating-and-:}).

We now correct the two aforementioned small errors in the above reasoning. In the second step, the optional hollow dot above the first entry is actually required; otherwise the permutation starts with its biggest element and is not simple. Furthermore, when this hollow dot is inflated by entries (which spawn more hollow dots in the third cell), it is forbidden to have the optional hollow dot to its left, as this violates the definition of greediness. In Figures~\ref{figure:bad-stuff-1}~and~\ref{figure:bad-stuff-2}, the required hollow dot is represented by a triangle, while the forbidden hollow dot is represented by a square. Thankfully, neither of these issues occur after the third cell.

\begin{figure}
		\hfill\hfill
		\minipage{0.5\textwidth}
			\begin{center}
				\begin{tikzpicture}[scale=.3, yscale=-1]
					\foreach \x/\y in {0/0, 5/0}{
						\draw[very thick, color=lightgray] (\x, \y) rectangle (\x + 5, \y + 5);
					}
					\draw (3,3)--(8,3);
					\draw (1,1)--(6,1);
					\node[closed] at (2,2) {};
					\node[closed] at (4,4) {};
					\node[open2, regular polygon,regular polygon sides=3] at (6,1) {};
					\node[open2] at (8,3) {};
					\node[white] at (5,10) {};
				\end{tikzpicture}
			\end{center}
			\caption{The hollow triangle represents the location of the hollow dot which is required.}
			\label{figure:bad-stuff-1}
		\endminipage\hfill
		\minipage{0.5\textwidth}
			\begin{center}
				\begin{tikzpicture}[scale=.3, yscale=-1]
					\foreach \x/\y in {0/0, 5/0, 5/5}{
						\draw[very thick, color=lightgray] (\x, \y) rectangle (\x + 5, \y + 5);
					}
					\draw (5.9,1)--(5.9,5.9);
					\draw (7.6,2.6)--(7.6,7.6);
					\node[closed] at (2,2) {};
					\node[closed] at (4,4) {};
					\node[zball, minimum size=.15cm] at (5.9,5.9) {};
					\node[closed] at (6.4,1.4) {};
					\node[open2] at (7.6,7.6) {};
					\node[closed] at (8.4,3.4) {};
				\end{tikzpicture}
			\end{center}
			\caption{The hollow square represents the location of the hollow dot which is forbidden.}
			\label{figure:bad-stuff-2}
		\endminipage
		\hfill\hfill
\end{figure}

The recurrence can be fixed by modifying $s_2$ to require that the uppermost hollow dot is added; moreover, we represent it by a $z$ instead of $y$ so that we can handle it separately in the third cell. Thus we have
	\[s_2(x,y,z) = xz + x^2yz + x^3y^2z + \cdots = \f{xz}{1-xy}.\]
To prevent adding the leftmost optional hollow dot in the third cell, we substitute $z = x + x^2y + x^3y^2 + \cdots$, so that
	\[s_3(x,y) = s_2\pa{x, \f{x(1+y)}{1-xy}, \f{x}{1-xy}}.\]
The correct generating function is then found as before by taking the limit as $n \to \infty$ of the recurrence
	\[s_{n+1}(x,y) = s_n\pa{x, \f{x(1+y)}{1-xy}}.\]

Whereas the recurrence above started in the northwest corner and proceeded southeast, our recurrence will instead begin in the ``middle'' of the permutation and proceed outward in two directions simultaneously so that the permutations at every intermediate step are involutions. When we place a hollow dot outside of the initial cell, we of course also must place its inverse image in a different cell. However, in the generating functions that we build, we count only the first of these hollow dots. Then, to build an involution, when we fill each hollow dot with permutation entries in the next step of the recurrence, each such entry is counted by $x^2$ to account for both the entry and its inverse image. In other words, hollow dots on one side of the fixed point are ignored until they become permutation entries. Due to this, the substitution $y = x(y+1)/(1-xy)$ used in~\cite{albert:generating-and-:} becomes $y = x^2(y+1)/(1-x^2y)$.

\begin{figure}
  \centering
  \begin{tikzpicture}[scale=.15]

    % draw the outer boxes, using a loop
    \foreach \x/\y in {0/0}
    \draw[very thick, color=lightgray] 
    (\x, \y) -- (\x + 10, \y) -- (\x + 10, \y + 10) -- (\x, \y + 10) -- cycle;
    % dotted line through diagonal
    \draw[dashed] (0,0) -- (10,10);

    % draw the center y dot
    \node[open_y] at (5,5) {};
    \node[fill=none] at (5,-12.5) {$s^{(1)}_1$};
  \end{tikzpicture}
  \hspace{2em}
  \begin{tikzpicture}[scale=.15]
    % draw the outer boxes, using a loop
    \foreach \x/\y in {-10/0, 0/0, 0/-10}
    \draw[very thick, color=lightgray] 
    (\x, \y) -- (\x + 10, \y) -- (\x + 10, \y + 10) -- (\x, \y + 10) -- cycle;
    % dashed line through diagonal
    \draw[dashed] (-9,-9) -- (11,11);

    % closed dots
    \foreach \i in {-4, -2, 0, 2, 4}{
      \node[closed] at (5-\i,5+\i) {};
    }
    % open dots
    \foreach \i in {1,3}{
      \draw (-5-\i,5+\i) -- (5.5-\i, 5+\i);
      \node[open] at (-5-\i, 5+\i) {};
    }
    \foreach \i in {-1,-3}{
      \draw (5-\i,-5+\i) -- (5-\i, 5.5+\i);
      \node[open_y] at (5-\i, -5+\i) {};
    }
    \draw (2,-2) -- (2, 8.5);
    \draw (-2,2) -- (8.5,2);
    \node[zball] at (2,-2) {};
    \node[zball] at (-2,2) {};

    \node[fill=none] at (5,-12.5) {$s^{(1)}_2$};
  \end{tikzpicture}
  \hspace{2em}
  \begin{tikzpicture}[scale=.15] % s3
    % draw the outer boxes, using a loop
    \foreach \x/\y in {-10/10, -10/0, 0/0, 0/-10, 10/-10}
    \draw[very thick, color=lightgray] 
    (\x, \y) -- (\x + 10, \y) -- (\x + 10, \y + 10) -- (\x, \y + 10) -- cycle;
    % dashed line through diagonal
    \draw[dashed] (-9,-9) -- (18,18);

    % closed dots
    \foreach \i in {-4, -2, 0, 2, 4}{
      \node[closed] at (5-\i,5+\i) {};
    }
    \foreach \i in {1,-3.3,-2.7, 2.7,3.3}{
      \node[closed] at (-5-\i, 5+\i) {};
    }
    \foreach \i in {2.7,3.3, -1,-2.7, -3.3}{
      \node[closed] at (5-\i, -5+\i) {};
    }

    % open dots
    \foreach \i in {-3, 3}{
      \draw (15-\i, -4.8+\i) -- (5.3-\i, -4.8+\i);
      \draw (-4.8+\i, 15-\i) -- (-4.8+\i, 5.3-\i);
      \node[open_y] at (15-\i, -4.8+\i) {};
      \node[open] at (-4.8+\i, 15-\i) {};
      
    }
    % extra open dot
    \draw (15,-5) -- (5.5, -5);
    \draw (-5,15) -- (-5, 5.5);
    \node[open_y] at (15,-5) {};
    \node[open] at (-5,15) {};
    
    \node[fill=none] at (5,-12.5) {$s^{(1)}_3$};

  \end{tikzpicture}

  \caption{Three stages of the recurrence, in the case when the single fixed point is a right-to-left maximum.}
  \label{fig-three-stages-10}
\end{figure}
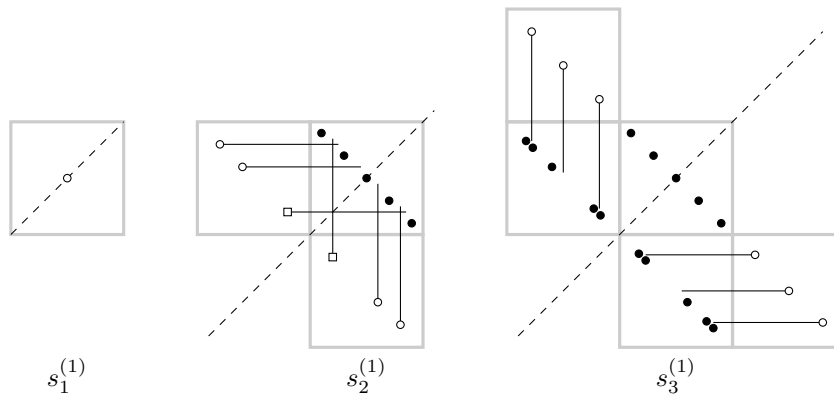

We begin by giving a detailed analysis of the case where $\sigma$ has precisely one fixed point (and hence must be of odd length). This analysis is accompanied by Figure~\ref{fig-three-stages-10}. We first find the generating function $s^{(1)}(x)$ which counts these permutations by length alone, and then refine it to obtain $\widehat{s}^{(1)}(u,v)$.

	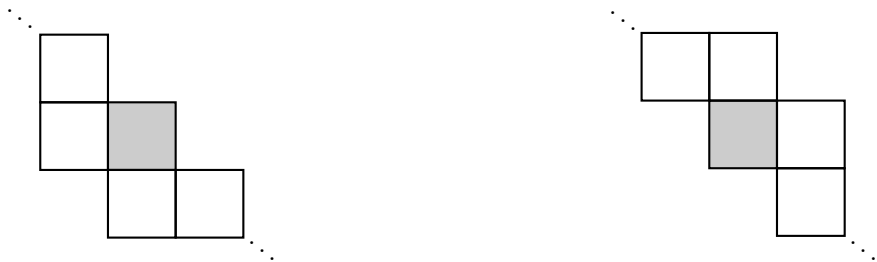
\begin{figure}
	\centering
		\begin{subfigure}[b]{0.43\textwidth}
			\begin{center}
				\begin{tikzpicture}[scale=.45]
					\draw[thick,fill=lightgray] (0,0) rectangle (2,2);
					\draw[thick] (-2,0) rectangle (0,2);
					\draw[thick] (0,0) rectangle (2,-2);
					\draw[thick] (-2,2) rectangle (0,4);
					\draw[thick] (2,-2) rectangle (4,0);
					\node at (-2.6,4.7) {$\ddots$};
					\node at (4.55,-2.15) {$\ddots$};
				\end{tikzpicture}
			\end{center}	
			\label{figure:one-fixed-point-shapes-1}
		\end{subfigure}
		\hspace*{.08\textwidth}
		\begin{subfigure}[b]{0.43\textwidth}
			\begin{center}
				\begin{tikzpicture}[scale=.45,cm={-1,0,0,-1,(0,0)}]
					\draw[thick,fill=lightgray] (0,0) rectangle (2,2);
					\draw[thick] (-2,0) rectangle (0,2);
					\draw[thick] (0,0) rectangle (2,-2);
					\draw[thick] (-2,2) rectangle (0,4);
					\draw[thick] (2,-2) rectangle (4,0);
					\node at (-2.55,4.2) {$\ddots$};
					\node at (4.55,-2.6) {$\ddots$};
				\end{tikzpicture}
			\end{center}	
			\label{figure:one-fixed-point-shapes-2}
		\end{subfigure}
		\caption{The diagrams on which we can draw simple permutations $\sigma \in \Av^I(123)$ that contain a single fixed point, depending on whether the fixed point is a right-to-left maximum or a left-to-right maximum. The starting point of the recurrence is the shaded cell.}
		\label{figure:one-fixed-point-shapes}
	\end{figure}

The single fixed point in $\sigma$ may be either a right-to-left maximum or a left-to-right minimum, but not both because $\sigma$ is not skew decomposable. These two cases are depicted in Figure~\ref{figure:one-fixed-point-shapes}. Because reflection across the anti-diagonal is a bijection between these two cases, we may restrict our attention to the case where the fixed point is a right-to-left maximum (and then multiply by $2$ at the end). The simplest way to define greediness in this context is to focus exclusively on the cells below and to the right of the initial cell (which has already been defined). Greediness for these cells is defined analogously to the non-involution case, and determines greediness for the cells above and to the left of the initial cell.

Suppose that the initial cell (which contains the fixed point) contains a total of $2k+1$ entries. It follows that $k$ of these entries lie below and to the right of the fixed point. Because $\sigma$ is simple, each of the $2k$ adjacent pairs of entries in this cell must be separated by entries in the cell below, in cell to the left, or by entries in both locations. Each adjacent pair lying above and to the left of the fixed point has a corresponding adjacent pair (its image under inversion) which lies below and to the right of the fixed point; if we split the former to the left, then the inverse image of the separating entry splits the latter below, and vice versa.

We can split all $2k$ of these adjacent pairs with as few as $2k$ hollow dots, $k$ in the cell below and $k$ in the cell to the left. The number of ways to split these pairs in this minimal way is $2^k$, because it suffices to choose which of each two corresponding pairs of entries is split below. Of course these adjacent pairs can also be split using more hollow dots. In general, the number of ways to have $k+i$ hollow dots in the cell below is given by $2^{k-i}{k \choose i}$, since we first choose which of the $i$ corresponding pairs of gaps between entries are split both to the left and below, then we choose which of each of the remaining $k-i$ corresponding pairs are split below.

The only problem we can have in this construction is if the resulting permutation is skew decomposable. This occurs precisely when we split precisely the pairs of entries to the right of the fixed point by entries below the initial cell, as shown in Figure~\ref{figure:bad-case}. We compensate for these ``bad cases'' by subtracting the term $x/(1 - x^2y)$.

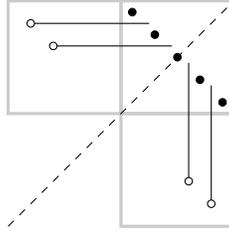
\begin{figure}
\centering
  \begin{tikzpicture}[scale=.15]
    % draw the outer boxes, using a loop
    \foreach \x/\y in {-10/0, 0/0, 0/-10}{
      \draw[very thick, color=lightgray] (\x, \y) rectangle (\x + 10, \y + 10);
    }
    % dotted line through diagonal
    \draw[dashed] (-10,-10) -- (10,10);

    % closed dots
    \foreach \i in {-4, -2, 0, 2, 4}{
      \node[closed] at (5-\i,5+\i) {};
    }
    % open dots
    \foreach \i in {1,3}{
      \draw (-5-\i,5+\i) -- (5.5-\i, 5+\i);
      \node[open] at (-5-\i, 5+\i) {};
    }
    \foreach \i in {-1,-3}{
      \draw (5-\i,-5+\i) -- (5-\i, 5.5+\i);
      \node[open_y] at (5-\i, -5+\i) {};
    }
  \end{tikzpicture}
  \caption{An example of a bad placement of splitting entries that leads to a skew decomposable permutation.}
  \label{figure:bad-case}
\end{figure}

It follows that
	\begin{eqnarray*}
	s_2^{(1)}(x,y,z)
	&=&
	\frac{2z}{y}\left(\sum_{k=0}^\infty \left(x^{2k+1}\sum_{i=0}^k 2^{k-i}{k\choose i}y^{k+i}\right) - \frac{x}{1-x^2y}\right)\\
	&=&
	\frac{2x^3z(1+y)}{(1-x^2y)(1-2x^2y-x^2y^2)}.
	\end{eqnarray*}
	% s2 := 2*z*(x^2*y + x^2) / (x^4*y^3 + (2*x^4 - x^2)*y^2 - 3*x^2*y + 1);

The $2$ in $s_2^{(1)}$ accounts for both cases shown in Figure~\ref{figure:one-fixed-point-shapes}, while the $z/y$ factor counts the topmost hollow dot in the cell below the fixed point by $z$ instead of $y$. By our definition of greediness, this topmost hollow dot is not allowed to lie below an entry to its right in the next cell. Therefore, when substituting for $z$ to obtain $s_3^{(1)}$, we substitute $x^2/(1-x^2y)$ instead of $x^2(1+y)/(1-x^2y)$. As such, we obtain

	\[
	s_3^{(1)}(x,y)=s_2^{(1)}\left(x, \frac{x^2(1+y)}{1-x^2y}, \frac{x^2}{1-x^2y}\right).
	% s3 := subs({y = x^2*(1+y) / (1 - x^2*y), z = x*(1 - x^2*y)}, s2);
	\]

After finding $s_3^{(1)}$, all later $s_n^{(1)}$ are easy to compute:
	\[
	s_{n+1}^{(1)}(x,y)=s_n^{(1)}\left(x,\frac{x^2(y+1)}{1-x^2y}\right).
	\]
	
To find $s^{(1)}$, we substitute the fixed point
	\[
	y=\frac{1-x^2-\sqrt{1-2x^2-3x^4}}{2x^2},
	% s := subs(y = (1 - x^2 - sqrt(1-2*x^2-3*x^4)) / 2*x^2, s3);
	\]
which satisfies $y=x^2(y+1)/(1-x^2y)$ into $s_3^{(1)}$, obtaining
	\begin{eqnarray*}
	s^{(1)}(x) &=& \frac{2x^5(1+x^2+\sqrt{1 - 2x^2 - 3x^4})}{
	        (1+x^2)^2 \left( 1 -3x^2 + (1-2x^2) \sqrt{1 - 2x^2 - 3x^4}\right)} \\
	  &=& 2x^5 + 2x^7 + 10x^9 + 22x^{11} + 68x^{13} + 184x^{15} + 530x^{17} +
	        1502x^{19} + \cdots .
	\end{eqnarray*}

Next we refine $s^{(1)}$ to count simple $123$-avoiding involutions with a single fixed point by their number of left-to-right minima and right-to-left maxima. We assume that the fixed point is a right-to-left maximum, and so all entries in the first cell are right-to-left maxima and the entries in the cells directly below and to the left are left-to-right minima. We call the generating function for these involutions $\widehat{s}^{(1)}(u,v)$; the generating function for the case where the fixed point is a left-to-right minimum is given by $\widehat{s}^{(1)}(v,u)$. Adjusting our formula above, using $u$ and $y_u$ to represent filled and hollow dots (respectively) which are left-to-right minima and $v$ and $y_v$ to represent filled and hollow dots (respectively) which are right-to-left maxima, we obtain 
	\begin{eqnarray*}
	\widehat{s}_2^{(1)}(u,v,y_u,y_v,z)
	&=&
	\frac{z}{y_u}\left(\sum_{k=0}^\infty \left(v^{2k+1}\sum_{i=0}^k 2^{k-i}{k\choose i}y_u^{k+i}\right) - \frac{v}{1-v^2y_u}\right)\\
	&=&
	\frac{v^3z(y_u+1)}{(1-v^2y_u)(1-2v^2y_u-v^2y_u^2)}.
	\end{eqnarray*}
Occurrences of $y_u$ in this generating function account for hollow dots which will become left-to-right minima when filled. Thus $\widehat{s}_3^{(1)}$ is obtained by the substitution
	\[
	\widehat{s}_3^{(1)}(u,v,y_u,y_v)
	=
	\widehat{s}_2^{(1)}\left(u,v,\frac{u^2(1+y_v)}{1-u^2y_v},\frac{v^2(1+y_u)}{1-v^2y_u},\frac{u^2}{1-u^2y_v}\right).
	\]
% s3 := subs({yu = u^2*(1+yv) / (1-u^2*yv), z = u^2/(1-u^2*yv)}, s2);
The general substitution rule is
	\[
	\widehat{s}_{n+1}^{(1)}(u,v,y_u,y_v)
	=
	\widehat{s}_n^{(1)}\left(u,v,\frac{u^2(1+y_v)}{1-u^2y_v},\frac{v^2(1+y_u)}{1-v^2y_u}\right).
	\]
Note that $\widehat{s}^{(1)}_3$ depends only on $u$, $v$, and $y_v$. Hence, to compute $\widehat{s}^{(1)}$ we need only substitute the solution of
	\[
	y_v
	=
	\frac{v^2(1+y_u)}{1-v^2y_u}
	=
	\frac{v^2(1+u^2)}{1-u^2v^2-u^2y_v-u^2v^2y_v}
	\]
into $\widehat{s}_3^{(1)}$. This solution is
	\[
	y_v
	=
	\frac{1-u^2v^2-\sqrt{1-6u^2v^2-4u^2v^4-4u^4v^2-3u^4v^4}}{2u^2(1+v^2)},
% yv := (1 - u^2*v^2 - sqrt(1 - 6*u^2*v^2 - 4*u^2*v^4 - 4*u^4*v^2 - 3*u^4*v^4)) / (2*u^2*(1+v^2));
	\]
and upon substituting this into $\widehat{s}_3^{(1)}$ we obtain
	\[
	\widehat{s}^{(1)}(u,v)
	=
	\frac{u^2v^3(1+u^2)(1+2v^2+u^2v^2+r)}
	{(1+v^2)(1-6u^2v^2-4u^2v^4-4u^4v^2-3u^4v^4+(1-3u^2v^2-2u^4v^2)r)}
	%
	% r := sqrt(1-6*u^2*v^2-4*u^2*v^4-4*u^4*v^2-3*u^4*v^4);
	% s := (u^2*v^3*(1+u^2)*(1+2*v^2+u^2*v^2+r)) / ((1+v^2)*(1-6*u^2*v^2-4*u^2*v^4-4*u^4*v^2-3*u^4*v^4+(1-3*u^2*v^2-2*u^4*v^2)*r));
	%
	\]
where
	\[
	r=\sqrt{1-6u^2v^2-4u^2v^4-4u^4v^2-3u^4v^4}.
	\]

For the next two computations, of $s^{(0)}$ and $s^{(2)}$, we explain only how to count the permutations according to length, as the refinements of these enumerations to count left-to-right minima and right-to-left maxima are analogous to the case of $s^{(1)}$. When enumerating $s^{(1)}$ we assumed (by symmetry) that the fixed point was a right-to-left maximum, and thus our initial cell was an upper-right corner of the staircase decomposition. When counting $s^{(0)}$ we of course do not have a fixed point, and for $s^{(2)}$ one fixed point will be a right-to-left maximum while the other will be a left-to-right minimum, so we must adjust this convention. In the case of $s^{(0)}$ we assume that the diagonal line about which inversion reflects the permutation passes through the middle of an upper-right corner of the staircase decomposition, and we take this cell to be our initial cell. In the case of $s^{(2)}$, we choose our initial cell in this case to contain the right-to-left maximum fixed point, and slightly tweak our conventions by considering the other fixed point to also lie in the southwest corner of this cell.

We begin by showing that
	\begin{eqnarray*}
	s^{(0)}_2(x,y,z)
	&=&
	\frac{z}{y}\left(\sum_{k=1}^\infty\left(x^{2k}\sum_{i=0}^{k-1} 2^{k-i-1}{k-1\choose i}y^{k+i}\right) - \frac{x^2y}{1-x^2y}\right)\\
	&=&
	\frac{x^4yz(1+y)}{(1-x^2y)(1-2x^2y-x^2y^2)}.
	% s2 := x^4*y*z*(1+y)/(1-x^2*y)/(1-x^2y^2-2*x^2*y);
	% s3 := subs(y=x^2*(1+y)/(1-x^2*y), z=x^2/(1-x^2*y), s2);
	% s := subs(y = (1 - x^2 - sqrt(1-2*x^2-3*x^4)) / (2*x^2), s3);
	\end{eqnarray*}
The justification of this equality is almost identical to the justification of $s^{(1)}_2$; in particular the same adjustment is required to avoid producing a skew decomposable permutation. Adjacent pairs are split in the same manner as the $s^{(1)}_2$ case, except that for $s^{(0)}_2$ we must separate the central adjacent pair, and because we are building an involution, this means that it must be split both by an entry in the cell below and by an entry in the cell to the left. (We do not need to multiply by 2 because in the case of zero (or two) fixed points, every such permutation can be drawn on both shapes in Figure~\ref{figure:one-fixed-point-shapes}.) We then produce subsequent $s^{(0)}_n$ and thus also $s^{(0)}(x)$ in the same manner as the one-fixed-point case, leading to
	\begin{eqnarray*}
	s^{(0)}(x)
	&=&
	\frac{2x^6(1+x^2-\sqrt{1-2x^2-3x^4})}{2-2x^2-10x^4-6x^6+(2-6x^4-4x^6)\sqrt{1-2x^2-3x^4}}\\
	% s := (2*x^6 + 2*x^8 - 2*x^6*sqrt(1-2*x^2-3*x^4)) / (2-2*x^2-10*x^4-6*x^6+(2-6*x^4-4*x^6)*sqrt(1-2*x^2-3*x^4))
	&=&
	x^8+2x^{10}+8x^{12}+22x^{14}+68x^{16}+198x^{18}+586x^{20}+\cdots.
	\end{eqnarray*}
The corresponding bivariate generating function is then
	\[
	\widehat{s}^{(0)}(u,v)
	=
	\frac{2u^2v^4(1+u^2)(1+2u^2+u^2v^2-r)}
	{(1-u^2v^2+r)(1-6u^2v^2-4u^2v^4-4u^4v^2-3u^4v^4+(1+2v^2+u^2v^2)r)},
	\]
	%
	% r := sqrt(1-6*u^2*v^2-4*u^2*v^4-4*u^4*v^2-3*u^4*v^4);
	% s := (2*u^2*v^4*(1+u^2)*(1+2*u^2+u^2*v^2-r)) / ((1-u^2*v^2+r)*(1-6*u^2*v^2-4*u^2*v^4-4*u^4*v^2-3*u^4*v^4+(1+2*v^2+u^2*v^2)*r));
	%
where $r$ is the same radical as before.

Finally, we move on to counting simple $123$-avoiding involutions with two fixed points. The presence of a second fixed point allows us to include the former ``bad cases'' which we had previously excluded, as the resulting permutation will no longer be skew decomposable. Furthermore, the entry which was previously marked by $z$ to prevent adding an entry above it in the cell to its right is now allowed to have this entry because the fixed point which is a left-to-right minimum prevents this added point from violating greediness.

Instead, we are now allowed to insert an entry in the cell below the initial cell which lies immediately to the right of the leftmost fixed point. However, if we choose to insert such an entry then we are required have an entry above it in the cell to its right (to prevent the inserted entry from forming an interval with the leftmost fixed point). So, if this entry exists we mark it by $w$ instead of $y$. Then, in calculating $s^{(2)}_3$, we substitute $x^2y/(1-x^2y)$ for $w$ instead of $x^2(1+y)/(1-x^2y)$. It follows that $s^{(2)}_2$ is counted by the generating function
	\begin{eqnarray*}
	s^{(2)}_2(x,y,w)
	&=&
	x^2(1+w)\left(\sum_{k=0}^\infty x^{2k}\sum_{i=0}^k 2^{k-i}{k\choose i}y^{k+i}\right)-x^2\\
	&=&
	\frac{x^2(w+2x^2y+x^2y^2)}{1-2x^2y-x^2y^2}.
	\end{eqnarray*}
while $s^{(2)}_3$ is
	\[
	s_3^{(2)}(x,y)=s_2^{(2)}\left(x, \frac{x^2(1+y)}{1-x^2y}, \frac{x^2y}{1-x^2y}\right).
	% s3 := subs({y = x^2*(1+y) / (1 - x^2*y), z = x*(1 - x^2*y)}, s2);
	\]
All later $s_n^{(2)}$ and then $s^{(2)}$ itself are produced as in the previous cases, leading to
	\begin{eqnarray*}
	s^{(2)}(x)
	&=&
	\frac{x^4(2+5x^2+3x^4-(2+x^2)\sqrt{1-2x^2-3x^4})}
	{1-x^2-5x^4-3x^6+(1+2x^2+x^4)\sqrt{1-2x^2-3x^4}}\\
	&=&
	3x^6+4x^8+15x^{10}+36x^{12}+105x^{14}+288x^{16}+819x^{18}+\cdots.
	% s := (2*x^4+5*x^6+3*x^8-(2*x^4+x^6)*sqrt(1-2*x^2-3*x^4)) / (1-x^2-5*x^4-3*x^6+(1+2*x^2+x^4)*sqrt(1-2*x^2-3*x^4));
	\end{eqnarray*}
Accounting for right-to-left maxima and left-to-right minima separately, we obtain
	\[
	\widehat{s}^{(2)}(u,v)
	=
	\frac{uv^3\pa{2+7u^2+4u^2v^2+4u^4+3u^4v^2-(2+u^2)r}}
	{1-6u^2v^2-4u^2v^4-4u^4v^2-3u^4v^4+(1+2v^2+u^2v^2)r},
	\]
	%
	% s := -(3*u^4*v^2+4*u^4+4*u^2*v^2-(-3*u^4*v^4-4*u^2*v^4-4*u^4*v^2-6*u^2*v^2+1)^(1/2)*u^2+7*u^2-2*(-3*u^4*v^4-4*u^2*v^4-4*u^4*v^2-6*u^2*v^2+1)^(1/2)+2)*u*v^3/(3*u^4*v^4+4*u^4*v^2+4*u^2*v^4-(-3*u^4*v^4-4*u^2*v^4-4*u^4*v^2-6*u^2*v^2+1)^(1/2)*u^2*v^2+6*u^2*v^2-2*(-3*u^4*v^4-4*u^2*v^4-4*u^4*v^2-6*u^2*v^2+1)^(1/2)*v^2-(-3*u^4*v^4-4*u^2*v^4-4*u^4*v^2-6*u^2*v^2+1)^(1/2)-1)

where $r$ is the same radical as in the previous two cases.

\section{Involutions Avoiding $1342$}
\label{sec-1342}

Obviously, every involution avoiding $1342$ must also avoid $1342^{-1}=1423$. Our first goal in this section is to show that the simple permutations in the class $\Av(1342, 1423)$ all avoid $123$. Thanks to a result in the literature, we are able to prove this result quite easily. Given a class $\C$, we define its \emph{substitution closure}, $\langle\C\rangle$, to be the largest class with the same simple permutations as $\C$.

In their investigation of substitution closures of principal classes, Atkinson, Ru\v{s}kuc, and Smith~\cite{atkinson:substitution-cl:} showed that very few of these substitution closures are finitely based.  Fortunately for us, $\langle\Av(123)\rangle$ is an exception: by bounding the length of potential basis elements of this class and then conducting an exhaustive computer search, they established that
\[
	\left\langle\Av(123)\right\rangle
	=
	\Av(24153, 25314, 31524, 41352, 246135, 415263).
\]
By inspection, it is clear that each of the basis elements of $\langle\Av(123)\rangle$ contains either $1342$ or $1423$, and thus $\Av(1342,1423)\subseteq\langle\Av(123)\rangle$. Going in the other direction, it follows trivially that every $123$-avoiding simple permutation avoids both $1342$ and $1423$. Thus we have the following result.

\begin{proposition}
\label{prop-1342-simple-involutions}
The simple permutations of $\Av(1342,1423)$ are precisely the same as the simple permutations of $\Av(123)$.
\end{proposition}

To enumerate the set $\Av^I(1342)$, we combine the generating functions of Section~\ref{sec-123} with the unique decompositions introduced in Section~\ref{sec-simples}. We retain our conventions from Section~\ref{sec-simples} by defining $f$ to be the generating function for the class $\Av(1342,1423)$ and $f_\oplus$ (resp., $f_\ominus$) the generating function for the sum (resp., skew) decomposable permutations of this class. We then define $g$ to be the generating function for the set $\Av^I(1342)$ and $g_\oplus$ (resp., $g_\ominus$) the generating function for the sum (resp., skew) decomposable $1342$-avoiding involutions.

First we describe the sum decomposable permutations $\pi=\alpha_1\oplus\alpha_2$ counted by $g_\oplus$. By Proposition~\ref{simple-decomp-unique}, we can assure the uniqueness of this decomposition by requiring that $\alpha_1$ is sum indecomposable. To produce an involution, $\alpha_1$ and $\alpha_2$ must be involutions as well. In order for $\pi$ to avoid the patterns $1342$ and $1423$, it is necessary and sufficient that $\alpha_1$ avoids these patterns and that $\alpha_2$ avoids the patterns $231$ and $312=231^{-1}$.

In fact, the class $\Av(231,312)$, known as the class of \emph{layered permutations}, consists entirely of involutions because a permutation lies in $\Av(231,312)$ if and only if it can be expressed as a sum of some number of decreasing permutations. The layered permutations of length $n$ are in bijection with compositions of $n$, and hence there are $2^{n-1}$ permutations of length $n$ in $\Av(231,312)$. Therefore, $g_\oplus$ satisfies the equation
	\[g_\oplus = \pa{g - g_\oplus}\pa{\f{x}{1-2x}},\]
from which it follows that
	\begin{equation}
	\label{eq-1342-1}
	g_\oplus = \f{gx}{1-x}.
	\end{equation}

Next we must briefly consider the class $\Av(1342,1423)$. Kremer~\cite{kremer:permutations-wi:, kremer:postscript:-per:} showed that this class is counted by the large Schr\"oder numbers, \OEIS{A006318}, and has generating function
	\[f(x) = \f{1-x-\sqrt{1-6x+x^2}}{2}.\]
Since the class $\Av(1342,1423)$ is skew closed (because both $1342$ and $1423$ are skew indecomposable), it follows by Proposition~\ref{involution-decomp-2} that
	\[f_\ominus = (f - f_\ominus)f,\]
and thus
	\[f_\ominus = \f{f^2}{1+f},\]
so
	\[f - f_\ominus = \f{f}{1+f} = \f{1+x-\sqrt{1-6x+x^2}}{4},\]
the generating function for the small Schr\"oder numbers, \OEIS{A001003}.

Returning to $\Av^I(1342)$, we see that skew decomposable permutations in this set are of the form $\alpha_1\ominus\alpha_2\ominus\alpha_1^{-1}$ where $\alpha_1$ is a skew indecomposable member of $\Av(1342,1423)$ and $\alpha_2$ is an arbitrary (and possibly empty) member of $\Av^I(1342)$. Therefore we see that
	\begin{equation}
	\label{eq-1342-2}
	g_\ominus = \pa{f(x^2)-f_\ominus(x^2)}(1+g).
	\end{equation}
	
Lastly, we must enumerate $1342$-avoiding involutions which are inflations of simple permutations of length at least four. Any such simple permutation must have at least two right-to-left maxima and by simplicity every right-to-left maximum must have some entry both below it and to the left. Hence to avoid creating a copy of $1342$ or $1423$, we may only inflate right-to-left maxima by decreasing intervals. An entry which is a left-to-right minimum can be inflated by any permutation in the class $\Av(1342,1432)$. However, to ensure that the inflated permutation is an involution, we must inflate each fixed point by an involution. Additionally, if we inflate the entry with value $\sigma(i)$ by the permutation $\alpha$, we must make sure to inflate the entry with value $i$ by $\alpha^{-1}$.

Consider $\widehat{s}^{(0)}(u,v)$, which is the generating function for simple involutions of length at least four which avoid $123$ and have zero fixed points. To inflate each right-to-left maximum by a decreasing permutation in a way that yields an involution, we substitute
	\[v^2 = \f{x^2}{1-x^2},\]
because if $\sigma(i)$ is a right-to-left maximum of the simple $123$-avoiding involution $\sigma$ then the entry with value $i$ will also be a right-to-left maximum, and we must substitute a permutation and its inverse into this pair of entries of $\sigma$. Because the class $\Av(1342,1423)$ is counted by the large Schr\"oder numbers, the inflations of the simple involutions of length at least four with zero fixed points are counted by
	\begin{equation}
	\label{eq-1342-3}
	\eval{\widehat{s}^{(0)}(u,v)}_{u^2=f(x^2),\;v^2=x^2/(1-x^2)}.
	\end{equation}
	
Recall that $\widehat{s}^{(1)}(u,v)$ counts only those simple involutions whose single fixed point is a right-to-left maximum. Since this fixed point must be inflated by a decreasing permutation, we count inflations of such permutations by
	\begin{equation}
	\label{eq-1342-4}
	\pa{\eval{\f{\widehat{s}^{(1)}(u,v)}{v}}_{u^2=f(x^2),\;v^2=x^2/(1-x^2)}}
	\cdot\f{x}{1-x}.
	\end{equation}
To count those simple involutions whose single fixed point is a left-to-right minimum, we need only swap $u$ and $v$. Thus, inflations of these are counted by the generating function
	\begin{equation}
	\label{eq-1342-5}
	\pa{\eval{\f{\widehat{s}^{(1)}(v,u)}{u}}_{u^2=f(x^2),\;v^2=x^2/(1-x^2)}}\cdot g.
	\end{equation}
Finally, we must account for inflations of those simple involutions which contain exactly two fixed points, one of which is a right-to-left maximum while the other is a left-to-right minimum. These permutations are counted by 
	\begin{equation}
	\label{eq-1342-6}
	\pa{\eval{\f{\widehat{s}^{(2)}(u,v)}{uv}}_{u^2=f(x^2),\;v^2=x^2/(1-x^2)}}\cdot\f{gx}{1-x}.
	\end{equation}
By summing the contributions of \eqref{eq-1342-1}--\eqref{eq-1342-6} and accounting for the single permutation of length $1$, one finds that
	\[g(x) = \f{x\pa{1-2x+x^2+\sqrt{1-6x^2+x^4}}}{2\pa{1-3x+x^2}}.\]
	% a := x*(1-2*x+x^2+sqrt(1-6*x^2 + x^4))/(2*(1-3*x+x^2));
It can then be computed that the growth rate of involutions avoiding $1342$ is $1$ plus the golden ratio,
	\[1+\f{1+\sqrt{5}}{2} \approx 2.62.\]

\section{Involutions Avoiding $2341$}
\label{sec-2341}

Again we begin by noting that every involution avoiding $2341$ must also avoid $2341^{-1}=4123$. Unlike the case when avoiding $1342$, the simple permutations of $\Av(2341,4123)$ are a proper superset of the simple permutations of $\Av(123)$. However, when we restrict our attention to involutions, there is only one simple involution which avoids $2341$ and $4123$ but contains $123$.

\begin{theorem}
\label{thm-2341-simple-involutions}
The simple $2341$-avoiding involutions consist exactly of the permutation $5274163$ along with the simple involutions of the set $\Av^I(123)$.
\end{theorem}
\begin{proof}
To prove this statement, we must consider several possible cases relating to the fixed points of a $2341$-avoiding simple involution. To assist in visualizing these arguments, we depict permutations by using \emph{permutation diagrams}, which consist of a permutation plotted on top of a grid of cells. A cell is white if we are allowed to insert a new entry into that cell without creating an occurrence of $2341$, a cell is light gray if we specifically forbid any entries to be placed in that cell, and a cell is dark gray if inserting an entry into that cell would create an occurrence of $2341$. 
	
	Define the \emph{rectangular hull} of a set $S$ of points (in our case, entries of a permutation plotted on the plane) to be the smallest axis-parallel rectangle which contains all points of $S$.
	
	Let $\sigma$ be a $2341$-avoiding simple involution and note that $\sigma$ must also avoid $2341^{-1} = 4123$. If $\sigma \in \Av(123)$, then there is nothing to prove. Thus, suppose that $123 \leq \sigma$. Choose the occurrence of $123$ for which the `$3$' is topmost possible entry, the `$1$' is the bottommost possible entry for the chosen `$3$', and the `$2$' is the rightmost possible entry for the chosen `$1$' and `$3$'. Let $i < j < k$ be the position of these entries, so that the entries forming the chosen $123$ pattern are $\sigma(i) < \sigma(j) < \sigma(k)$.
	
	By these assumptions, $\sigma$ can be drawn on the permutation diagram shown in Figure~\ref{figure:6_1-1}. Note that the three shown entries, from left to right, are $\sigma(i)$, $\sigma(j)$, and $\sigma(k)$. Though this figure makes it seem like these three entries are all fixed points, this is misleading. Because each white cell could contain many entries, we must consider separate cases in which some combination of these three entries are fixed points. We will show that in one case we must have $\sigma = 5274163$, and in the other cases no such $\sigma$ can exist.
	
	\bigskip
	
	\textbf{Case 1:} \emph{$\sigma(i)$, $\sigma(j)$, and $\sigma(k)$ are all fixed points}
	
	Assume that $\sigma(i)$, $\sigma(j)$, and $\sigma(k)$ are all fixed points. Then, since $\sigma$ is an involution, there cannot be any entries in cells $A$, $B$, or $C$. To see this, suppose there were an entry in cell $A$, for example. Then, in order for $\sigma$ to be an involution, there must be an entry in cell $\widebar{A}$. It must again be noted that this argument is only valid because $\sigma(i)$ is assumed in this case to be a fixed point. Therefore, $\sigma$ has the permutation diagram shown in Figure~\ref{figure:6_1-2}. 
	
	\begin{figure}
	\begin{footnotesize}
	\centering
		\hfill\hfill
		\begin{subfigure}[b]{0.13\textwidth}
			\begin{center}
				%Tikz output
\definecolor{light-gray}{gray}{0.65}
\definecolor{dark-gray}{gray}{0.35}\begin{tikzpicture}[scale=.45]
%Forbidden regions
\filldraw[light-gray](0,0) rectangle (1,1);
\filldraw[dark-gray](0,3) rectangle (1,4);
\filldraw[light-gray](1,0) rectangle (2,1);
\filldraw[light-gray](2,1) rectangle (3,2);
\filldraw[light-gray](2,2) rectangle (3,3);
\filldraw[light-gray](2,3) rectangle (3,4);
\filldraw[dark-gray](3,0) rectangle (4,1);
\filldraw[light-gray](3,3) rectangle (4,4);
%Points
\draw[black, fill=black] (1,1) circle (0.2);
\draw[black, fill=black] (2,2) circle (0.2);
\draw[black, fill=black] (3,3) circle (0.2);
%Gridlines
\draw[thick](0,0)--(0,4);
\draw[thick](0,0)--(4,0);
\draw[thick](1,0)--(1,4);
\draw[thick](0,1)--(4,1);
\draw[thick](2,0)--(2,4);
\draw[thick](0,2)--(4,2);
\draw[thick](3,0)--(3,4);
\draw[thick](0,3)--(4,3);
\draw[thick](4,0)--(4,4);
\draw[thick](0,4)--(4,4);
\node at (.5,1.5) {$A$};
\node at (1.5,2.5) {$B$};
\node at (3.5,2.5) {$C$};
\node at (1.5,.5) {$\widebar{A}$};
\end{tikzpicture}

			\end{center}	
			\caption{}
			\label{figure:6_1-1}
		\end{subfigure}
		\hfill
		\begin{subfigure}[b]{0.13\textwidth}
			\begin{center}
				%Tikz output
\definecolor{light-gray}{gray}{0.65}
\definecolor{dark-gray}{gray}{0.35}\begin{tikzpicture}[scale=.45]
%Forbidden regions
\filldraw[light-gray](0,0) rectangle (1,1);
\filldraw[light-gray](0,1) rectangle (1,2);
\filldraw[dark-gray](0,3) rectangle (1,4);
\filldraw[light-gray](1,0) rectangle (2,1);
\filldraw[light-gray](1,2) rectangle (2,3);
\filldraw[light-gray](2,1) rectangle (3,2);
\filldraw[light-gray](2,2) rectangle (3,3);
\filldraw[light-gray](2,3) rectangle (3,4);
\filldraw[dark-gray](3,0) rectangle (4,1);
\filldraw[light-gray](3,2) rectangle (4,3);
\filldraw[light-gray](3,3) rectangle (4,4);
%Points
\draw[black, fill=black] (1,1) circle (0.2);
\draw[black, fill=black] (2,2) circle (0.2);
\draw[black, fill=black] (3,3) circle (0.2);
%Gridlines
\draw[thick](0,0)--(0,4);
\draw[thick](0,0)--(4,0);
\draw[thick](1,0)--(1,4);
\draw[thick](0,1)--(4,1);
\draw[thick](2,0)--(2,4);
\draw[thick](0,2)--(4,2);
\draw[thick](3,0)--(3,4);
\draw[thick](0,3)--(4,3);
\draw[thick](4,0)--(4,4);
\draw[thick](0,4)--(4,4);
\end{tikzpicture}

			\end{center}	
			\caption{}
			\label{figure:6_1-2}
		\end{subfigure}
		\hfill
		\begin{subfigure}[b]{0.3\textwidth}
			\begin{center}
				%Tikz output
\definecolor{light-gray}{gray}{0.65}
\definecolor{dark-gray}{gray}{0.35}\begin{tikzpicture}[scale=.45]
%Forbidden regions
\filldraw[light-gray](0,0) rectangle (1,1);
\filldraw[dark-gray](0,1) rectangle (1,2);
\filldraw[dark-gray](0,2) rectangle (1,3);
\filldraw[dark-gray](0,3) rectangle (1,4);
\filldraw[dark-gray](0,4) rectangle (1,5);
\filldraw[light-gray](0,5) rectangle (1,6);
\filldraw[dark-gray](0,6) rectangle (1,7);
\filldraw[dark-gray](0,7) rectangle (1,8);
\filldraw[dark-gray](1,0) rectangle (2,1);
\filldraw[dark-gray](1,1) rectangle (2,2);
\filldraw[light-gray](1,2) rectangle (2,3);
\filldraw[dark-gray](1,3) rectangle (2,4);
\filldraw[dark-gray](1,5) rectangle (2,6);
\filldraw[dark-gray](1,6) rectangle (2,7);
\filldraw[dark-gray](1,7) rectangle (2,8);
\filldraw[dark-gray](2,0) rectangle (3,1);
\filldraw[light-gray](2,1) rectangle (3,2);
\filldraw[dark-gray](2,2) rectangle (3,3);
\filldraw[dark-gray](2,3) rectangle (3,4);
\filldraw[dark-gray](2,4) rectangle (3,5);
\filldraw[dark-gray](2,5) rectangle (3,6);
\filldraw[dark-gray](2,6) rectangle (3,7);
\filldraw[light-gray](2,7) rectangle (3,8);
\filldraw[dark-gray](3,0) rectangle (4,1);
\filldraw[dark-gray](3,1) rectangle (4,2);
\filldraw[dark-gray](3,2) rectangle (4,3);
\filldraw[dark-gray](3,3) rectangle (4,4);
\filldraw[light-gray](3,4) rectangle (4,5);
\filldraw[dark-gray](3,5) rectangle (4,6);
\filldraw[dark-gray](3,7) rectangle (4,8);
\filldraw[dark-gray](4,0) rectangle (5,1);
\filldraw[dark-gray](4,2) rectangle (5,3);
\filldraw[light-gray](4,3) rectangle (5,4);
\filldraw[dark-gray](4,4) rectangle (5,5);
\filldraw[dark-gray](4,5) rectangle (5,6);
\filldraw[dark-gray](4,6) rectangle (5,7);
\filldraw[dark-gray](4,7) rectangle (5,8);
\filldraw[light-gray](5,0) rectangle (6,1);
\filldraw[dark-gray](5,1) rectangle (6,2);
\filldraw[dark-gray](5,2) rectangle (6,3);
\filldraw[dark-gray](5,3) rectangle (6,4);
\filldraw[dark-gray](5,4) rectangle (6,5);
\filldraw[dark-gray](5,5) rectangle (6,6);
\filldraw[light-gray](5,6) rectangle (6,7);
\filldraw[dark-gray](5,7) rectangle (6,8);
\filldraw[dark-gray](6,0) rectangle (7,1);
\filldraw[dark-gray](6,1) rectangle (7,2);
\filldraw[dark-gray](6,2) rectangle (7,3);
\filldraw[dark-gray](6,4) rectangle (7,5);
\filldraw[light-gray](6,5) rectangle (7,6);
\filldraw[dark-gray](6,6) rectangle (7,7);
\filldraw[dark-gray](6,7) rectangle (7,8);
\filldraw[dark-gray](7,0) rectangle (8,1);
\filldraw[dark-gray](7,1) rectangle (8,2);
\filldraw[light-gray](7,2) rectangle (8,3);
\filldraw[dark-gray](7,3) rectangle (8,4);
\filldraw[dark-gray](7,4) rectangle (8,5);
\filldraw[dark-gray](7,5) rectangle (8,6);
\filldraw[dark-gray](7,6) rectangle (8,7);
\filldraw[light-gray](7,7) rectangle (8,8);
%Points
\draw[black, fill=black] (1,5) circle (0.2);
\draw[black, fill=black] (2,2) circle (0.2);
\draw[black, fill=black] (3,7) circle (0.2);
\draw[black, fill=black] (4,4) circle (0.2);
\draw[black, fill=black] (5,1) circle (0.2);
\draw[black, fill=black] (6,6) circle (0.2);
\draw[black, fill=black] (7,3) circle (0.2);
%Gridlines
\draw[thick](0,0)--(0,8);
\draw[thick](0,0)--(8,0);
\draw[thick](1,0)--(1,8);
\draw[thick](0,1)--(8,1);
\draw[thick](2,0)--(2,8);
\draw[thick](0,2)--(8,2);
\draw[thick](3,0)--(3,8);
\draw[thick](0,3)--(8,3);
\draw[thick](4,0)--(4,8);
\draw[thick](0,4)--(8,4);
\draw[thick](5,0)--(5,8);
\draw[thick](0,5)--(8,5);
\draw[thick](6,0)--(6,8);
\draw[thick](0,6)--(8,6);
\draw[thick](7,0)--(7,8);
\draw[thick](0,7)--(8,7);
\draw[thick](8,0)--(8,8);
\draw[thick](0,8)--(8,8);
\end{tikzpicture}

			\end{center}	
			\caption{}
			\label{figure:6_1-3}
		\end{subfigure}
		\hfill\hfill
		\caption{Permutation diagrams corresponding to Case 1 in the proof of Theorem~\ref{thm-2341-simple-involutions}.}
		\label{figure:6_1-group-1}
	\end{footnotesize}
	\end{figure}
	
	Since $\sigma$ is simple, the rectangular hull of $\sigma(j)$ and $\sigma(k)$ has a separating entry in either the white cell above it or the white cell to its right. The fact that $\sigma$ is an involution forces there to be splitting entries in both of these cells. We choose to plot the splitting entry in the white cell above this rectangular hull that is the topmost possible entry and the splitting entry in the white cell to the right of this rectangular hull that is the rightmost possible entry. At the same time, the rectangular hull of $\sigma(j)$ and $\sigma(k)$ must similarly be split both below and to the left. This gives the permutation diagram depicted in Figure~\ref{figure:6_1-3}.
	
	There are only four remaining cells where entries can be inserted. However, since no two of these cells share a row or column any entry in one of these cells would be part of a proper interval, contradicting the simplicity of $\sigma$. This shows that in this case the only permutation that can be obtained is $5274163$.
	
	\bigskip
	
	\textbf{Case 2:} \emph{$\sigma(k)$ is not a fixed point}
	
	Suppose that $\pi(k)$ is not a fixed point. It must lie either above or below the ``reflection line''. In other words, it must either be either above and to the left of its inverse image or below and to the right of its inverse image. Suppose first that it is below the reflection line. Then, it must have its inverse image somewhere above and to its left. We see that there is only one place to put such an entry. The result is shown in Figure~\ref{figure:6_1-4}.
	
	We next observe a general fact about involutions.
	
	\begin{fact}
	\label{fact:inv-img}
	If two entries form an inversion (resp., a non-inversion), then their inverse images also form an inversion (resp., a non-inversion).
	\end{fact}
	
	For this reason, the third entry from the left shown in Figure~\ref{figure:6_1-4} (which was the `$2$' in the original $123$) cannot lie above the reflection line, nor can it be a fixed point, from which it follows that this entry lies below the reflection line. Its inverse image can only lie in one particular white cell, as shown in Figure~\ref{figure:6_1-5}. If the leftmost entry in this figure were a fixed point, then the permutation would start with the entry $1$, violating simplicity. Therefore, this entry has an inverse image above it and to its left. This yields Figure~\ref{figure:6_1-6}. If the entries contained in the six white cells in the bottom-left corner of Figure~\ref{figure:6_1-6} did not form an interval, then they would have to be split by either an entry above them or an entry to their right; either splitting would create a copy of $2341$ or $4123$. Therefore we can eliminate this case.
		
	\begin{figure}
	\begin{footnotesize}
	\centering
		\begin{subfigure}[b]{0.2\textwidth}
			\begin{center}
				%Tikz output
\definecolor{light-gray}{gray}{0.65}
\definecolor{dark-gray}{gray}{0.35}\begin{tikzpicture}[scale=.4]
%Forbidden regions
\filldraw[light-gray](0,0) rectangle (1,1);
\filldraw[dark-gray](0,3) rectangle (1,4);
\filldraw[dark-gray](0,4) rectangle (1,5);
\filldraw[light-gray](1,0) rectangle (2,1);
\filldraw[dark-gray](2,0) rectangle (3,1);
\filldraw[dark-gray](2,1) rectangle (3,2);
\filldraw[light-gray](3,1) rectangle (4,2);
\filldraw[dark-gray](3,2) rectangle (4,3);
\filldraw[light-gray](3,3) rectangle (4,4);
\filldraw[light-gray](3,4) rectangle (4,5);
\filldraw[dark-gray](4,0) rectangle (5,1);
\filldraw[dark-gray](4,3) rectangle (5,4);
\filldraw[light-gray](4,4) rectangle (5,5);
%Points
\draw[black, fill=black] (1,1) circle (0.2);
\draw[black, fill=black] (2,4) circle (0.2);
\draw[black, fill=black] (3,2) circle (0.2);
\draw[black, fill=black] (4,3) circle (0.2);
%Gridlines
\draw[thick](0,0)--(0,5);
\draw[thick](0,0)--(5,0);
\draw[thick](1,0)--(1,5);
\draw[thick](0,1)--(5,1);
\draw[thick](2,0)--(2,5);
\draw[thick](0,2)--(5,2);
\draw[thick](3,0)--(3,5);
\draw[thick](0,3)--(5,3);
\draw[thick](4,0)--(4,5);
\draw[thick](0,4)--(5,4);
\draw[thick](5,0)--(5,5);
\draw[thick](0,5)--(5,5);
\end{tikzpicture}

			\end{center}	
			\caption{}
			\label{figure:6_1-4}
		\end{subfigure}
		\hfill
		\begin{subfigure}[b]{0.2\textwidth}
			\begin{center}
				%Tikz output
\definecolor{light-gray}{gray}{0.65}
\definecolor{dark-gray}{gray}{0.35}\begin{tikzpicture}[scale=.4]
%Forbidden regions
\filldraw[light-gray](0,0) rectangle (1,1);
\filldraw[dark-gray](0,2) rectangle (1,3);
\filldraw[dark-gray](0,3) rectangle (1,4);
\filldraw[dark-gray](0,4) rectangle (1,5);
\filldraw[dark-gray](0,5) rectangle (1,6);
\filldraw[light-gray](1,0) rectangle (2,1);
\filldraw[dark-gray](1,2) rectangle (2,3);
\filldraw[dark-gray](1,3) rectangle (2,4);
\filldraw[dark-gray](2,0) rectangle (3,1);
\filldraw[dark-gray](2,1) rectangle (3,2);
\filldraw[dark-gray](2,4) rectangle (3,5);
\filldraw[dark-gray](3,0) rectangle (4,1);
\filldraw[dark-gray](3,1) rectangle (4,2);
\filldraw[dark-gray](3,5) rectangle (4,6);
\filldraw[dark-gray](4,0) rectangle (5,1);
\filldraw[light-gray](4,1) rectangle (5,2);
\filldraw[dark-gray](4,2) rectangle (5,3);
\filldraw[light-gray](4,3) rectangle (5,4);
\filldraw[light-gray](4,4) rectangle (5,5);
\filldraw[dark-gray](4,5) rectangle (5,6);
\filldraw[dark-gray](5,0) rectangle (6,1);
\filldraw[dark-gray](5,3) rectangle (6,4);
\filldraw[dark-gray](5,4) rectangle (6,5);
\filldraw[light-gray](5,5) rectangle (6,6);
%Points
\draw[black, fill=black] (1,1) circle (0.2);
\draw[black, fill=black] (2,4) circle (0.2);
\draw[black, fill=black] (3,5) circle (0.2);
\draw[black, fill=black] (4,2) circle (0.2);
\draw[black, fill=black] (5,3) circle (0.2);
%Gridlines
\draw[thick](0,0)--(0,6);
\draw[thick](0,0)--(6,0);
\draw[thick](1,0)--(1,6);
\draw[thick](0,1)--(6,1);
\draw[thick](2,0)--(2,6);
\draw[thick](0,2)--(6,2);
\draw[thick](3,0)--(3,6);
\draw[thick](0,3)--(6,3);
\draw[thick](4,0)--(4,6);
\draw[thick](0,4)--(6,4);
\draw[thick](5,0)--(5,6);
\draw[thick](0,5)--(6,5);
\draw[thick](6,0)--(6,6);
\draw[thick](0,6)--(6,6);
\end{tikzpicture}

			\end{center}	
			\caption{}
			\label{figure:6_1-5}
		\end{subfigure}
		\hfill
		\begin{subfigure}[b]{0.2\textwidth}
			\begin{center}
				%Tikz output
\definecolor{light-gray}{gray}{0.65}
\definecolor{dark-gray}{gray}{0.35}\begin{tikzpicture}[scale=.4]
%Forbidden regions
\filldraw[light-gray](0,0) rectangle (1,1);
\filldraw[dark-gray](0,3) rectangle (1,4);
\filldraw[dark-gray](0,4) rectangle (1,5);
\filldraw[dark-gray](0,5) rectangle (1,6);
\filldraw[dark-gray](0,6) rectangle (1,7);
\filldraw[light-gray](1,0) rectangle (2,1);
\filldraw[dark-gray](1,3) rectangle (2,4);
\filldraw[dark-gray](1,4) rectangle (2,5);
\filldraw[dark-gray](1,5) rectangle (2,6);
\filldraw[dark-gray](1,6) rectangle (2,7);
\filldraw[light-gray](2,0) rectangle (3,1);
\filldraw[dark-gray](2,3) rectangle (3,4);
\filldraw[dark-gray](2,4) rectangle (3,5);
\filldraw[dark-gray](3,0) rectangle (4,1);
\filldraw[dark-gray](3,1) rectangle (4,2);
\filldraw[dark-gray](3,2) rectangle (4,3);
\filldraw[dark-gray](3,5) rectangle (4,6);
\filldraw[dark-gray](4,0) rectangle (5,1);
\filldraw[dark-gray](4,1) rectangle (5,2);
\filldraw[dark-gray](4,2) rectangle (5,3);
\filldraw[dark-gray](4,6) rectangle (5,7);
\filldraw[dark-gray](5,0) rectangle (6,1);
\filldraw[dark-gray](5,1) rectangle (6,2);
\filldraw[light-gray](5,2) rectangle (6,3);
\filldraw[dark-gray](5,3) rectangle (6,4);
\filldraw[light-gray](5,4) rectangle (6,5);
\filldraw[light-gray](5,5) rectangle (6,6);
\filldraw[dark-gray](5,6) rectangle (6,7);
\filldraw[dark-gray](6,0) rectangle (7,1);
\filldraw[dark-gray](6,1) rectangle (7,2);
\filldraw[dark-gray](6,4) rectangle (7,5);
\filldraw[dark-gray](6,5) rectangle (7,6);
\filldraw[light-gray](6,6) rectangle (7,7);
%Points
\draw[black, fill=black] (1,2) circle (0.2);
\draw[black, fill=black] (2,1) circle (0.2);
\draw[black, fill=black] (3,5) circle (0.2);
\draw[black, fill=black] (4,6) circle (0.2);
\draw[black, fill=black] (5,3) circle (0.2);
\draw[black, fill=black] (6,4) circle (0.2);
%Gridlines
\draw[thick](0,0)--(0,7);
\draw[thick](0,0)--(7,0);
\draw[thick](1,0)--(1,7);
\draw[thick](0,1)--(7,1);
\draw[thick](2,0)--(2,7);
\draw[thick](0,2)--(7,2);
\draw[thick](3,0)--(3,7);
\draw[thick](0,3)--(7,3);
\draw[thick](4,0)--(4,7);
\draw[thick](0,4)--(7,4);
\draw[thick](5,0)--(5,7);
\draw[thick](0,5)--(7,5);
\draw[thick](6,0)--(6,7);
\draw[thick](0,6)--(7,6);
\draw[thick](7,0)--(7,7);
\draw[thick](0,7)--(7,7);
\end{tikzpicture}

			\end{center}	
			\caption{}
			\label{figure:6_1-6}
		\end{subfigure}
		\hfill
		\begin{subfigure}[b]{0.2\textwidth}
			\begin{center}
		%Tikz output
\definecolor{light-gray}{gray}{0.65}
\definecolor{dark-gray}{gray}{0.35}\begin{tikzpicture}[scale=.4]
%Forbidden regions
\filldraw[light-gray](0,0) rectangle (1,1);
\filldraw[dark-gray](0,3) rectangle (1,4);
\filldraw[dark-gray](0,4) rectangle (1,5);
\filldraw[light-gray](1,0) rectangle (2,1);
\filldraw[light-gray](2,1) rectangle (3,2);
\filldraw[light-gray](2,2) rectangle (3,3);
\filldraw[light-gray](2,3) rectangle (3,4);
\filldraw[light-gray](2,4) rectangle (3,5);
\filldraw[dark-gray](3,0) rectangle (4,1);
\filldraw[light-gray](3,4) rectangle (4,5);
\filldraw[dark-gray](4,0) rectangle (5,1);
\filldraw[light-gray](4,4) rectangle (5,5);
%Points
\draw[black, fill=black] (1,1) circle (0.2);
\draw[black, fill=black] (2,2) circle (0.2);
\draw[black, fill=black] (3,4) circle (0.2);
\draw[black, fill=black] (4,3) circle (0.2);
%Gridlines
\draw[thick](0,0)--(0,5);
\draw[thick](0,0)--(5,0);
\draw[thick](1,0)--(1,5);
\draw[thick](0,1)--(5,1);
\draw[thick](2,0)--(2,5);
\draw[thick](0,2)--(5,2);
\draw[thick](3,0)--(3,5);
\draw[thick](0,3)--(5,3);
\draw[thick](4,0)--(4,5);
\draw[thick](0,4)--(5,4);
\draw[thick](5,0)--(5,5);
\draw[thick](0,5)--(5,5);
\node at (4.5,3.5) {$D$};
\node at (3.5,4.5) {$\widebar{D}$};
\node at (3.5,2.5) {$E$};
\node at (2.5,3.5) {$\widebar{E}$};
\node at (1.5,3.5) {$\widebar{E}'$};
\end{tikzpicture}
\hfill\hfill
			\end{center}	
			\caption{}
			\label{figure:6_1-7a}
		\end{subfigure}
				\caption{Permutation diagrams corresponding to Case 2 in the proof of Theorem~\ref{thm-2341-simple-involutions}.}
		\label{figure:6_1-group-2}
	\end{footnotesize}
	\end{figure}

		\begin{figure}
	\centering
				\hfill
		\begin{subfigure}[b]{0.3\textwidth}
			\begin{center}
				%Tikz output
\definecolor{light-gray}{gray}{0.65}
\definecolor{dark-gray}{gray}{0.35}\begin{tikzpicture}[scale=.4]
%Forbidden regions
\filldraw[light-gray](0,0) rectangle (1,1);
\filldraw[dark-gray](0,3) rectangle (1,4);
\filldraw[dark-gray](0,4) rectangle (1,5);
\filldraw[light-gray](1,0) rectangle (2,1);
\filldraw[light-gray](2,1) rectangle (3,2);
\filldraw[light-gray](2,2) rectangle (3,3);
\filldraw[light-gray](2,3) rectangle (3,4);
\filldraw[light-gray](2,4) rectangle (3,5);
\filldraw[dark-gray](3,0) rectangle (4,1);
\filldraw[light-gray](3,2) rectangle (4,3);
\filldraw[light-gray](3,4) rectangle (4,5);
\filldraw[dark-gray](4,0) rectangle (5,1);
\filldraw[light-gray](4,3) rectangle (5,4);
\filldraw[light-gray](4,4) rectangle (5,5);
%Points
\draw[black, fill=black] (1,1) circle (0.2);
\draw[black, fill=black] (2,2) circle (0.2);
\draw[black, fill=black] (3,4) circle (0.2);
\draw[black, fill=black] (4,3) circle (0.2);
%Gridlines
\draw[thick](0,0)--(0,5);
\draw[thick](0,0)--(5,0);
\draw[thick](1,0)--(1,5);
\draw[thick](0,1)--(5,1);
\draw[thick](2,0)--(2,5);
\draw[thick](0,2)--(5,2);
\draw[thick](3,0)--(3,5);
\draw[thick](0,3)--(5,3);
\draw[thick](4,0)--(4,5);
\draw[thick](0,4)--(5,4);
\draw[thick](5,0)--(5,5);
\draw[thick](0,5)--(5,5);
\end{tikzpicture}
			\end{center}	
			\caption{}
			\label{figure:6_1-7}
		\end{subfigure}
		\hfill
		\begin{subfigure}[b]{0.3\textwidth}
			\begin{center}

				%Tikz output
\definecolor{light-gray}{gray}{0.65}
\definecolor{dark-gray}{gray}{0.35}\begin{tikzpicture}[scale=.4]
%Forbidden regions
\filldraw[light-gray](0,0) rectangle (1,1);
\filldraw[dark-gray](0,2) rectangle (1,3);
\filldraw[dark-gray](0,3) rectangle (1,4);
\filldraw[dark-gray](0,4) rectangle (1,5);
\filldraw[dark-gray](0,5) rectangle (1,6);
\filldraw[dark-gray](0,6) rectangle (1,7);
\filldraw[light-gray](1,0) rectangle (2,1);
\filldraw[dark-gray](1,2) rectangle (2,3);
\filldraw[dark-gray](1,3) rectangle (2,4);
\filldraw[dark-gray](1,4) rectangle (2,5);
\filldraw[dark-gray](2,0) rectangle (3,1);
\filldraw[dark-gray](2,1) rectangle (3,2);
\filldraw[dark-gray](2,2) rectangle (3,3);
\filldraw[dark-gray](2,5) rectangle (3,6);
\filldraw[dark-gray](3,0) rectangle (4,1);
\filldraw[dark-gray](3,1) rectangle (4,2);
\filldraw[light-gray](3,2) rectangle (4,3);
\filldraw[dark-gray](3,3) rectangle (4,4);
\filldraw[dark-gray](3,4) rectangle (4,5);
\filldraw[dark-gray](3,5) rectangle (4,6);
\filldraw[light-gray](3,6) rectangle (4,7);
\filldraw[dark-gray](4,0) rectangle (5,1);
\filldraw[dark-gray](4,1) rectangle (5,2);
\filldraw[dark-gray](4,3) rectangle (5,4);
\filldraw[dark-gray](4,6) rectangle (5,7);
\filldraw[dark-gray](5,0) rectangle (6,1);
\filldraw[dark-gray](5,2) rectangle (6,3);
\filldraw[dark-gray](5,3) rectangle (6,4);
\filldraw[dark-gray](5,6) rectangle (6,7);
\filldraw[dark-gray](6,0) rectangle (7,1);
\filldraw[dark-gray](6,4) rectangle (7,5);
\filldraw[dark-gray](6,5) rectangle (7,6);
\filldraw[light-gray](6,6) rectangle (7,7);
%Points
\draw[black, fill=black] (1,1) circle (0.2);
\draw[black, fill=black] (2,5) circle (0.2);
\draw[black, fill=black] (3,3) circle (0.2);
\draw[black, fill=black] (4,6) circle (0.2);
\draw[black, fill=black] (5,2) circle (0.2);
\draw[black, fill=black] (6,4) circle (0.2);
%Gridlines
\draw[thick](0,0)--(0,7);
\draw[thick](0,0)--(7,0);
\draw[thick](1,0)--(1,7);
\draw[thick](0,1)--(7,1);
\draw[thick](2,0)--(2,7);
\draw[thick](0,2)--(7,2);
\draw[thick](3,0)--(3,7);
\draw[thick](0,3)--(7,3);
\draw[thick](4,0)--(4,7);
\draw[thick](0,4)--(7,4);
\draw[thick](5,0)--(5,7);
\draw[thick](0,5)--(7,5);
\draw[thick](6,0)--(6,7);
\draw[thick](0,6)--(7,6);
\draw[thick](7,0)--(7,7);
\draw[thick](0,7)--(7,7);
\end{tikzpicture}
			\end{center}	
			\caption{}
			\label{figure:6_1-8}
		\end{subfigure}
		\hfill\hfill
				\caption{Permutation diagrams corresponding to Case 2 in the proof of Theorem~\ref{thm-2341-simple-involutions}.}
		\label{figure:6_1-group-2b}
	\end{figure}

	We have shown so far that we reach a contradiction if $\sigma(k)$ lies below the reflection line. Now suppose that this entry lies above the reflection line, so that it has an inverse image that lies below it and to the right. If this inverse image lies in the cell which is not immediately below it and to the right, then we end up with a permutation similar to that of Figure~\ref{figure:6_1-5}, yielding the same contradiction as in that case. Therefore, the inverse image of $\sigma(k)$ must lie in the white cell immediately below it and to the right, as in Figure~\ref{figure:6_1-7a}.
	
	Since $\sigma$ is an involution we can forbid placing entries into cells where their inverse image would lie in an already forbidden cell or into cells where their inverse image would create a forbidden pattern. For example, if cell $D$ in Figure~\ref{figure:6_1-7a} contained an entry, then by Fact~\ref{fact:inv-img} there must be an entry in cell $\widebar{D}$, which is not allowed. Additionally, if cell $E$ of the same figure contained an entry, then Fact~\ref{fact:inv-img} implies that either $\widebar{E}$ or $\widebar{E}'$ would contain the inverse image of this entry; each scenario is either forbidden by previous assumptions or else would imply the presence of a forbidden pattern. Hence, we have Figure~\ref{figure:6_1-7}. Now, to preserve simplicity the rectangular hull of the rightmost two entries must be split, and to preserve involutionhood it must in fact be split both below and to the left, yielding Figure~\ref{figure:6_1-8}. We are now in a situation similar to Figure~\ref{figure:6_1-6} in that any permutation built from this permutation diagram (with respect to involution and the forbidden patterns) will be sum decomposable and hence not simple. Hence, it is not possible that $\sigma(k)$ is not a fixed point.
	
	\bigskip
	
	\textbf{Case 3:} \emph{$\sigma(j)$ is not a fixed point}
	
	 If we assume that $\sigma(j)$ lies above the reflection line, then there are two possible locations for its inverse image. An application of Fact~\ref{fact:inv-img} leads us to a situation identical to Figure~\ref{figure:6_1-5} in one case and a situation identical to a 180 degree rotation of Figure~\ref{figure:6_1-5} in the other case. Therefore, $\sigma(j)$ must lie below the reflection line. For almost identical reasons, the inverse image of $\sigma(j)$ can lie in only one particular white cell. Figure~\ref{figure:6_1-9} shows the resulting image; note the several additional cells are greyed out because presence of entries in these cells would force the presence of an inverse image in an already forbidden cell.
	 
	\begin{figure}
	\begin{footnotesize}
	\centering
		\hfill
		\begin{subfigure}[b]{0.15\textwidth}
			\begin{center}
				%Tikz output
\definecolor{light-gray}{gray}{0.65}
\definecolor{dark-gray}{gray}{0.35}\begin{tikzpicture}[scale=.4]
%Forbidden regions
\filldraw[light-gray](0,0) rectangle (1,1);
\filldraw[dark-gray](0,4) rectangle (1,5);
\filldraw[light-gray](1,0) rectangle (2,1);
\filldraw[light-gray](2,0) rectangle (3,1);
\filldraw[light-gray](2,3) rectangle (3,4);
\filldraw[light-gray](3,1) rectangle (4,2);
\filldraw[light-gray](3,2) rectangle (4,3);
\filldraw[light-gray](3,3) rectangle (4,4);
\filldraw[light-gray](3,4) rectangle (4,5);
\filldraw[dark-gray](4,0) rectangle (5,1);
\filldraw[light-gray](4,3) rectangle (5,4);
\filldraw[light-gray](4,4) rectangle (5,5);
\filldraw[light-gray](1,3) rectangle (2,4);
%Points
\draw[black, fill=black] (1,1) circle (0.2);
\draw[black, fill=black] (2,3) circle (0.2);
\draw[black, fill=black] (3,2) circle (0.2);
\draw[black, fill=black] (4,4) circle (0.2);
%Gridlines
\draw[thick](0,0)--(0,5);
\draw[thick](0,0)--(5,0);
\draw[thick](1,0)--(1,5);
\draw[thick](0,1)--(5,1);
\draw[thick](2,0)--(2,5);
\draw[thick](0,2)--(5,2);
\draw[thick](3,0)--(3,5);
\draw[thick](0,3)--(5,3);
\draw[thick](4,0)--(4,5);
\draw[thick](0,4)--(5,4);
\draw[thick](5,0)--(5,5);
\draw[thick](0,5)--(5,5);
\end{tikzpicture}

			\end{center}	
			\caption{}
			\label{figure:6_1-9}
		\end{subfigure}
		\hfill
		\begin{subfigure}[b]{0.15\textwidth}
			\begin{center}
				%Tikz output
\definecolor{light-gray}{gray}{0.65}
\definecolor{dark-gray}{gray}{0.35}\begin{tikzpicture}[scale=.4]
%Forbidden regions
\filldraw[light-gray](0,0) rectangle (1,1);
\filldraw[light-gray](0,1) rectangle (1,2);
\filldraw[light-gray](0,2) rectangle (1,3);
\filldraw[dark-gray](0,4) rectangle (1,5);
\filldraw[light-gray](1,0) rectangle (2,1);
\filldraw[light-gray](1,3) rectangle (2,4);
\filldraw[light-gray](2,0) rectangle (3,1);
\filldraw[light-gray](2,3) rectangle (3,4);
\filldraw[light-gray](3,1) rectangle (4,2);
\filldraw[light-gray](3,2) rectangle (4,3);
\filldraw[light-gray](3,3) rectangle (4,4);
\filldraw[light-gray](3,4) rectangle (4,5);
\filldraw[dark-gray](4,0) rectangle (5,1);
\filldraw[light-gray](4,3) rectangle (5,4);
\filldraw[light-gray](4,4) rectangle (5,5);
%Points
\draw[black, fill=black] (1,1) circle (0.2);
\draw[black, fill=black] (2,3) circle (0.2);
\draw[black, fill=black] (3,2) circle (0.2);
\draw[black, fill=black] (4,4) circle (0.2);
%Gridlines
\draw[thick](0,0)--(0,5);
\draw[thick](0,0)--(5,0);
\draw[thick](1,0)--(1,5);
\draw[thick](0,1)--(5,1);
\draw[thick](2,0)--(2,5);
\draw[thick](0,2)--(5,2);
\draw[thick](3,0)--(3,5);
\draw[thick](0,3)--(5,3);
\draw[thick](4,0)--(4,5);
\draw[thick](0,4)--(5,4);
\draw[thick](5,0)--(5,5);
\draw[thick](0,5)--(5,5);
\end{tikzpicture}

			\end{center}	
			\caption{}
			\label{figure:6_1-10}
		\end{subfigure}
		\hfill
		\begin{subfigure}[b]{0.22\textwidth}
			\begin{center}
				%Tikz output
\definecolor{light-gray}{gray}{0.65}
\definecolor{dark-gray}{gray}{0.35}\begin{tikzpicture}[scale=.4]
%Forbidden regions
\filldraw[light-gray](0,0) rectangle (1,1);
\filldraw[dark-gray](0,1) rectangle (1,2);
\filldraw[light-gray](0,2) rectangle (1,3);
\filldraw[light-gray](0,3) rectangle (1,4);
\filldraw[dark-gray](0,6) rectangle (1,7);
\filldraw[dark-gray](1,0) rectangle (2,1);
\filldraw[dark-gray](1,1) rectangle (2,2);
\filldraw[light-gray](1,2) rectangle (2,3);
\filldraw[light-gray](1,3) rectangle (2,4);
\filldraw[dark-gray](1,6) rectangle (2,7);
\filldraw[light-gray](2,0) rectangle (3,1);
\filldraw[light-gray](2,1) rectangle (3,2);
\filldraw[dark-gray](2,2) rectangle (3,3);
\filldraw[dark-gray](2,3) rectangle (3,4);
\filldraw[light-gray](2,4) rectangle (3,5);
\filldraw[light-gray](2,5) rectangle (3,6);
\filldraw[light-gray](3,0) rectangle (4,1);
\filldraw[light-gray](3,1) rectangle (4,2);
\filldraw[dark-gray](3,2) rectangle (4,3);
\filldraw[dark-gray](3,4) rectangle (4,5);
\filldraw[dark-gray](3,5) rectangle (4,6);
\filldraw[dark-gray](3,6) rectangle (4,7);
\filldraw[light-gray](4,2) rectangle (5,3);
\filldraw[dark-gray](4,3) rectangle (5,4);
\filldraw[dark-gray](4,4) rectangle (5,5);
\filldraw[dark-gray](4,5) rectangle (5,6);
\filldraw[dark-gray](4,6) rectangle (5,7);
\filldraw[light-gray](5,2) rectangle (6,3);
\filldraw[dark-gray](5,3) rectangle (6,4);
\filldraw[dark-gray](5,4) rectangle (6,5);
\filldraw[light-gray](5,5) rectangle (6,6);
\filldraw[light-gray](5,6) rectangle (6,7);
\filldraw[dark-gray](6,0) rectangle (7,1);
\filldraw[dark-gray](6,1) rectangle (7,2);
\filldraw[dark-gray](6,3) rectangle (7,4);
\filldraw[dark-gray](6,4) rectangle (7,5);
\filldraw[light-gray](6,5) rectangle (7,6);
\filldraw[light-gray](6,6) rectangle (7,7);
%Points
\draw[black, fill=black] (1,5) circle (0.2);
\draw[black, fill=black] (2,2) circle (0.2);
\draw[black, fill=black] (3,4) circle (0.2);
\draw[black, fill=black] (4,3) circle (0.2);
\draw[black, fill=black] (5,1) circle (0.2);
\draw[black, fill=black] (6,6) circle (0.2);
%Gridlines
\draw[thick](0,0)--(0,7);
\draw[thick](0,0)--(7,0);
\draw[thick](1,0)--(1,7);
\draw[thick](0,1)--(7,1);
\draw[thick](2,0)--(2,7);
\draw[thick](0,2)--(7,2);
\draw[thick](3,0)--(3,7);
\draw[thick](0,3)--(7,3);
\draw[thick](4,0)--(4,7);
\draw[thick](0,4)--(7,4);
\draw[thick](5,0)--(5,7);
\draw[thick](0,5)--(7,5);
\draw[thick](6,0)--(6,7);
\draw[thick](0,6)--(7,6);
\draw[thick](7,0)--(7,7);
\draw[thick](0,7)--(7,7);
\end{tikzpicture}

			\end{center}	
			\caption{}
			\label{figure:6_1-11}
		\end{subfigure}
		\hfill
		\begin{subfigure}[b]{0.15\textwidth}
			\begin{center}
				%Tikz output
\definecolor{light-gray}{gray}{0.65}
\definecolor{dark-gray}{gray}{0.35}\begin{tikzpicture}[scale=.4]
%Forbidden regions
\filldraw[light-gray](0,0) rectangle (1,1);
\filldraw[light-gray](0,1) rectangle (1,2);
\filldraw[light-gray](0,2) rectangle (1,3);
\filldraw[light-gray](0,3) rectangle (1,4);
\filldraw[dark-gray](0,5) rectangle (1,6);
\filldraw[light-gray](1,0) rectangle (2,1);
\filldraw[light-gray](1,4) rectangle (2,5);
\filldraw[dark-gray](1,5) rectangle (2,6);
\filldraw[light-gray](2,0) rectangle (3,1);
\filldraw[light-gray](2,4) rectangle (3,5);
\filldraw[light-gray](3,0) rectangle (4,1);
\filldraw[light-gray](3,4) rectangle (4,5);
\filldraw[light-gray](4,1) rectangle (5,2);
\filldraw[light-gray](4,2) rectangle (5,3);
\filldraw[light-gray](4,3) rectangle (5,4);
\filldraw[light-gray](4,4) rectangle (5,5);
\filldraw[light-gray](4,5) rectangle (5,6);
\filldraw[dark-gray](5,0) rectangle (6,1);
\filldraw[dark-gray](5,1) rectangle (6,2);
\filldraw[light-gray](5,4) rectangle (6,5);
\filldraw[light-gray](5,5) rectangle (6,6);
%Points
\draw[black, fill=black] (1,2) circle (0.2);
\draw[black, fill=black] (2,1) circle (0.2);
\draw[black, fill=black] (3,4) circle (0.2);
\draw[black, fill=black] (4,3) circle (0.2);
\draw[black, fill=black] (5,5) circle (0.2);
%Gridlines
\draw[thick](0,0)--(0,6);
\draw[thick](0,0)--(6,0);
\draw[thick](1,0)--(1,6);
\draw[thick](0,1)--(6,1);
\draw[thick](2,0)--(2,6);
\draw[thick](0,2)--(6,2);
\draw[thick](3,0)--(3,6);
\draw[thick](0,3)--(6,3);
\draw[thick](4,0)--(4,6);
\draw[thick](0,4)--(6,4);
\draw[thick](5,0)--(5,6);
\draw[thick](0,5)--(6,5);
\draw[thick](6,0)--(6,6);
\draw[thick](0,6)--(6,6);
\node at (2.5, 1.5) {$G$};
\node at (3.5, 1.5) {$F$};
\node at (1.5, 2.5) {$\widebar{G}$};
\node at (1.5, 3.5) {$\widebar{F}$};
\end{tikzpicture}

			\end{center}	
			\caption{}
			\label{figure:6_1-12}
		\end{subfigure}
		\hfill
		\hfill
		\caption{Permutation diagrams corresponding to Case 3 in the proof of Theorem~\ref{thm-2341-simple-involutions}.}
		\label{figure:6_1-group-3}
	\end{footnotesize}
	\end{figure}

	If the leftmost entry shown in this diagram is a fixed point, then we have the permutation diagram in Figure~\ref{figure:6_1-10}. To maintain simplicity, there must be an entry in the bottommost white cell whose inverse image is in the leftmost white cell, yielding Figure~\ref{figure:6_1-11}. However, there is now an interval which cannot be split, which is a contradiction. Therefore, the leftmost entry cannot be a fixed point. If the leftmost entry were to lie above the reflection line, then its inverse image would have to be in the bottommost white square of Figure~\ref{figure:6_1-9}. However, this violates Fact~\ref{fact:inv-img} when comparing these two entries to the middle two entries shown in Figure~\ref{figure:6_1-9}. Thus the inverse image of this leftmost entry must lie above it and to the left, and for the same reason, can only lie in the white cell immediately above and to the left, yielding the permutation diagram in Figure~\ref{figure:6_1-12}. There are two ways that we could try to split the rectangular hull of the leftmost two entries. In the first way, we would place an entry in cell $F$ and its inverse image would lie in cell $\widebar{F}$, but this creates an occurrence of both forbidden patterns. In the second way, we would place an entry in cell $G$ and its inverse image would lie in cell $\widebar{G}$, but this creates a permutation which is necessary sum decomposable, and hence not simple.
	
	\textbf{Case 4:} \emph{$\sigma(i)$ is not a fixed point}
	
	When $\sigma(i)$ is not a fixed point, its inverse image lies above it and to the left. In order for $\sigma$ to not be sum decomposable, there must be an entry below and to the right of $\sigma(i)$, and this entry has an inverse image above and to the left of the inverse image of $\sigma(i)$.  This creates an unsplittable interval, similar to that of Figure~\ref{figure:6_1-11}. Hence this case is not possible, completing the proof.		
\end{proof}

As in the previous section, we enumerate the $2341$-avoiding involutions by separately enumerating the sum decomposable permutations, the skew decomposable permutations, and the inflations of simple permutations of length at least four. Again we define $g$ to be the generating function for the set $\Av^I(2341)$ and $g_\oplus$ (resp., $g_\ominus$) the generating function for the sum (resp., skew) decomposable $2341$-avoiding involutions.

In this case we see that $\Av^I(2341)$ is sum closed (in the sense that the sum of two $2341$-avoiding involutions must also be a $2341$-avoiding involution) and so we have
	\[g_\oplus = (g - g_\oplus)g,\]
and hence
	\begin{equation}
	\label{eq-2341-1}
	g_\oplus = \f{g^2}{1+g}.
	\end{equation}

By Proposition~\ref{involution-decomp-2}, the skew decomposable permutations must have the form $321[\alpha_1,\alpha_2,\alpha_1^{-1}]$, where $\alpha_1$ is skew indecomposable and $\alpha_2$ is a (possibly empty) involution. Furthermore, to avoid the occurrence of a $2341$ or a $4123$ pattern, we must also have that $\alpha_1,\alpha_2\in \Av(123)$.

The $123$-avoiding permutations are enumerated by the Catalan numbers, which have generating function
	\[c(x) = \f{1-2x-\sqrt{1-4x}}{2x} = x + 2x^2 + 5x^3 + 14x^4 + \cdots.\]
Let $c_\ominus$ denote the generating function for the skew decomposable $123$-avoiding permutations. Since the class $\Av(123)$ is skew closed, it follows that
	\[c_\ominus = c(c-c_\ominus),\]
and thus
	\[c-c_\ominus = \f{c}{1+c} = x(c+1).\]

As mentioned in the introduction, Simion and Schmidt~\cite{simion:restricted-perm:} proved that
	\[
	|\Av^I_n(123)|={n\choose \lfloor n/2\rfloor},
	\]
the central binomial coefficients, \OEIS{A001405}. These permutations thus have the generating function
	\[\frac{1-4x^2-\sqrt{1-4x^2}}{4x^2-2x} = x + 2x^2 + 3x^3 + 6x^4 + \cdots.\]
	% ( 1-4*x^2-sqrt(1-4*x^2) ) / ( 4*x^2-2*x )
Therefore, the generating function which counts our choices for the pair $(\alpha_1,\alpha_1^{-1})$ is $x^2(c(x^2)+1)$, and the generating function for all skew decomposable $2341$-avoiding involutions is
	\begin{equation}
	\label{eq-2341-2}
	g_\ominus
	=
	\pa{x^2\pa{c(x^2)+1}}
	\cdot
	\pa{\frac{1-4x^2-\sqrt{1-4x^2}}{4x^2-2x}+1}
	\end{equation}

Next, we consider inflations of the simple permutations in $\Av^I(123)$. By considering several cases, it can be shown that every entry of such a simple permutation can only be inflated by a decreasing permutation, as any inflation by a permutation with a non-inversion would create a copy of $2341$ and $4123$. Thus inflations of the simple permutations counted by $\widehat{s}^{(0)}$ contribute
	\begin{equation}
	\label{eq-2341-4}
	\eval{\widehat{s}^{(0)}(u,v)}_{u^2=v^2=x^2/(1-x^2)},
	\end{equation}
inflations of the simple permutations counted by $\widehat{s}^{(1)}$ contribute
	\begin{equation}
	\label{eq-2341-5}
	2\pa{\eval{\f{\widehat{s}^{(1)}(u,v)}{v}}_{u^2=v^2=x^2/(1-x^2)}}
	\cdot
	\f{x}{1-x},
	\end{equation}
and inflations of simple permutations counted by $\widehat{s}^{(2)}$ contribute
	\begin{equation}
	\label{eq-2341-6}
	\pa{\eval{\f{\widehat{s}^{(2)}(u,v)}{uv}}_{u^2=v^2=x^2/(1-x^2)}}
	\cdot
	\pa{\f{x}{1-x}}^2.
	\end{equation}

Lastly, we consider inflations of $5274163$. Again, inflation of any entry by a permutation containing a non-inversion creates an occurrence of both $2341$ and $4123$. Because this permutation has three fixed points, the $2341$-avoiding involutions formed by inflations of $5274163$ are counted by
	\begin{equation}
	\label{eq-2341-3}
	\pa{\f{x^2}{1-x^2}}^2\pa{\f{x}{1-x}}^3.
	\end{equation}

Combining the contributions of \eqref{eq-2341-1}--\eqref{eq-2341-6} and accounting for the single permutation of length $1$, we are able to solve for $g$:

\[g = \frac{(x+1)^4(x-1)^{10}\sqrt{1-4x^2} - p(x)}{2q(x)},\]
where
\begin{align*}
	p(x) =& \;\;1 - 8x + 17x^2 + 24x^3 - 151x^4 + 162x^5 + 221x^6 - 624x^7 + 231x^8 + 684x^9 - 801x^{10}\\
			&\;\;- 60x^{11} + 627x^{12} - 334x^{13} - 101x^{14} + 158x^{15} - 48x^{16},
\end{align*}
and
\begin{align*}
	q(x) =& \;\;1 - 6x + 4x^2 + 50x^3 - 141x^4 + 55x^5 + 326x^6 - 514x^7 - 26x^8 + 725x^9 - 561x^{10}\\
			&\;\; - 223x^{11} + 540x^{12} - 206x^{13} - 113x^{14} + 120x^{15} - 32x^{16}.
\end{align*}
	
From this we find that the growth rate of $\Av^I(2341,4123)$ is the reciprocal of the smallest (and only) positive real root of $q(x)$, approximately $2.54$.
%\begin{align*}
%x^{16} - 6x^{15} + 4x^{14} + 50x^{13} - 141x^{12} + 55x^{11} + 326x^{10} - 514x^9 - 26x^8 + 725x^7\\
%\qquad - 561x^6 - 223x^5 + 540x^4 - 206x^3 - 113x^2 + 120x - 32.
%% x^(16) - 6*x^(15) + 4*x^(14) + 50*x^(13) - 141*x^(12) + 55*x^(11) + 326*x^(10) - 514*x^9 - 26*x^8 + 725*x^7 - 561*x^6 - 223*x^5 + 540*x^4 - 206*x^3 - 113*x^2 + 120*x - 32
%\end{align*}
%The growth rate itself is approximately $2.54$.

% Actual generating function:
%-1/2*(48*x^16-158*x^15+101*x^14+334*x^13-627*x^12+60*x^11+801*x^10-684*x^9-231*x^8+624*x^7-221*x^6-162*x^5+151*x^4-24*x^3-17*x^2+(x+1)^4*(x-1)^10*(-4*x^2+1)^(1/2)+8*x-1)/(32*x^16-120*x^15+113*x^14+206*x^13-540*x^12+223*x^11+561*x^10-725*x^9+26*x^8+514*x^7-326*x^6-55*x^5+141*x^4-50*x^3-4*x^2+6*x-1)

\section{Involutions Avoiding $1324$ Revisited}\label{section:1324-revisited}

As stated in the introduction, we were initially interested in pattern-avoiding involutions because we noticed that $\gr(\Av^I(1324))>\gr(\Av^I(1234))$, but that the numbers in Table~\ref{table-enum-1} did not obey this relationship. The ratio between $|\Av^I_n(1234)|$ and $|\Av^I_n(1324)|$ is plotted on the left of Figure~\ref{fig-1324-2413}. Here we see that for large enough $n$, this ratio does indeed go below $1$ (and of course we know that it goes to $0$).

\begin{figure}
\begin{footnotesize}
\begin{center}
\begin{tabular}{ccc}

\begin{tikzpicture}[y=1.667cm, x=.2cm,font=\sffamily,scale=.8]
 	%axis
	\draw (0,0) -- coordinate (x axis mid) (25,0);
    	\draw (0,0) -- coordinate (y axis mid) (0,3);
    	%ticks
    	\foreach \x in {0,1,...,25}
     		\draw (\x,1pt) -- (\x,-3pt)
			node[anchor=north] {};
	\foreach \x in {0,5,...,25}
     		\draw (\x,1pt) -- (\x,-3pt)
			node[anchor=north] {\x};
    	\foreach \y in {0,1,2,3}
     		\draw (1pt,\y) -- (-3pt,\y) 
     			node[anchor=east] {\y}; 
	%labels      
	\node[right=6pt] at (25,3) {$\displaystyle\frac{|\operatorname{Av}_n^I(2413)|}{|\operatorname{Av}_n^I(1324)|}$};
	\node[right=6pt] at (25,.9) {$\displaystyle\frac{|\operatorname{Av}_n^I(1234)|}{|\operatorname{Av}_n^I(1324)|}$};

	%plots
	\draw [gray] (0,1)--(25,1);
	\draw plot[mark=*, mark size=.05cm] file {2413_invs.data};  
	\draw plot[mark=*, mark size=.05cm] file {1234_invs.data};  
    
\end{tikzpicture}
&&

\begin{tikzpicture}[y=5cm, x=.2cm,font=\sffamily,scale=.8]
 	%axis
	\draw (0,0) -- coordinate (x axis mid) (25,0);
    	\draw (0,0) -- coordinate (y axis mid) (0,1);
    	%ticks
    	\foreach \x in {0,1,...,25}
     		\draw (\x,1pt) -- (\x,-3pt)
			node[anchor=north] {};
	\foreach \x in {0,5,...,25}
     		\draw (\x,1pt) -- (\x,-3pt)
			node[anchor=north] {\x};
    	\foreach \y in {0,1}
     		\draw (1pt,\y) -- (-3pt,\y) 
     			node[anchor=east] {\y}; 
	%labels      
	\node[right=6pt] at (25,.7) {$\displaystyle\frac{|\operatorname{Av}_n(1234)|}{|\operatorname{Av}_n(1324)|}$};
	\node[right=6pt] at (25,.25) {$\displaystyle\frac{|\operatorname{Av}_n(2413)|}{|\operatorname{Av}_n(1324)|}$};

	%plots
	\draw [gray] (0,1)--(25,1);
	\draw plot[mark=*, mark size=.05cm] file {1234_perms.data};  
	\draw plot[mark=*, mark size=.05cm] file {1342_perms.data};  
    
\end{tikzpicture}
\end{tabular}
\end{center}
\end{footnotesize}
\caption{The number of $1324$-avoiding involutions and permutations compared to the number of $1234$- and $2413$-avoiding involutions and $1234$- and $1342$-avoiding permutations for $n=0$, $\dots$, $25$.} \label{fig-1324-2413}
\end{figure}
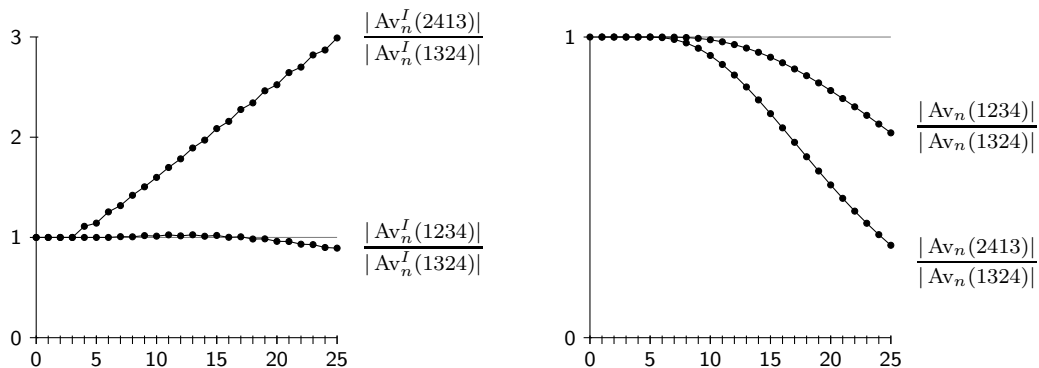

As observed in Section~\ref{section:grs}, the ratio between $|\Av^I_n(1234)|$ and $|\Av^I_n(1324)|$, which is also plotted on the left of Figure~\ref{fig-1324-2413}, should also go to $0$ (unless $\gr(\Av(1324))<9.9$, which would go against every bit of evidence we have about this class). However, as this plot demonstrates, the first $25$ terms of this ratio do not paint a very convincing picture of a sequence going to $0$. This is almost surely just an instance of the ``law of small numbers'', but it is interesting that the empirical data is so much worse for involutions than it is for permutations in general (as the analogous ratios shown on the right of Figure~\ref{fig-1324-2413} show).

For the rest of this section we adapt the technique of B\'ona~\cite{bona:a-new-upper-bou:} to derive an upper bound on the growth rate of $\Av^I(1324)$. B\'ona's technique was itself an improvement on the techniques of Claesson, Jel\'{\i}nek, and Steingr\'{\i}msson~\cite{claesson:upper-bounds-fo:}. They proved that every $1324$-avoiding permutation is the \emph{merge} of a $132$-avoiding permutation and a $213$-avoiding permutation. Here we say that $\pi$ is a merge of $\sigma$ and $\tau$ if the entries of $\pi$ can be partitioned into two subsequences such that one subsequence is order isomorphic to $\sigma$ while the other is order isomorphic to $\tau$. This gave an upper bound of $16$ on the growth rate of $\Av(1324)$.

Let $\pi \in \Av(1324)$. We color the entries of $\pi$ red or blue by the following algorithm. Proceeding from left to right, color an entry red only if it will not create a red $132$ pattern among the entries already colored. Otherwise, color it blue. The resulting coloring has the property that the red entries avoid $132$ and the blue entries avoid $213$. %Note also that all left-to-right minima are colored red. %Wait, this isn't true, since we might turn one blue later.

We now label each of the entries of $\pi$ by one of the four letters $\{\a,\b,\c,\d\}$ and use this to create two words, $e_\pi$ and $v_\pi$. A red entry is labeled $\a$ if it is a left-to-right minimum, and it is labeled $\b$ otherwise. Similarly, a blue entry is labeled $\d$ if it is a right-to-left maximum, and $\c$ otherwise. The $i$th letter of $e_\pi$ is then the label of $\pi(i)$ while the $i$th letter of $v_\pi$ is the label of the entry $i$ in $\pi$. B\'ona~\cite{bona:a-new-upper-bou:} proved that $\pi$ can be reconstructed from the pair $(e_\pi,v_\pi)$ and moreover, that neither $e_\pi$ nor $v_\pi$ can contain a $\c\b$ factor. Moreover, the generating function for words of length $n$ over the alphabet $\{\a,\b,\c,\d\}$ avoiding the factor $\c\b$ is
\[
	\frac{1}{1-4x+x^2},
\]
from which it follows that
\[
	\gr(\Av(1324))<\left(2+\sqrt{3}\right)^2=7+4\sqrt{3} < 13.93.
\]

Before adapting this technique to involutions, we alter the coloring algorithm slightly. Given a permutation $\pi \in \Av(1324)$, first color it as above. Then, change the color of all right-to-left maxima to blue. In order to show that the reconstruction given by B\'ona~\cite{bona:a-new-upper-bou:} still works, we must show that the blue entries still avoid $213$ (since we have not added any red entries, it is clear that the red entries still avoid $132$).

Assume to the contrary that there existed some right-to-left maximum $\pi(m)$ which was originally colored red, but is now part of a blue copy of $213$. Choose the leftmost such entry, and say that $\pi(a)$ is the `$1$' in the blue copy of $213$. Since we chose the leftmost $\pi(m)$, it must be true that $\pi(a)$ was chosen to be blue because otherwise it would be the `$2$' in a red copy of $132$. Let $\pi(x)$ and $\pi(y)$ be the entries that would have been the `$1$' and `$3$' (respectively) in such a red copy of $132$. If $\pi(y) < \pi(m)$, then the entries $\pi(x)\pi(y)\pi(a)\pi(m)$ form a $1324$ pattern, a contradiction. If $\pi(y) > \pi(m)$, then the entries $\pi(x)\pi(y)\pi(m)$ form a copy of $132$, which contradicts the assumption that $\pi(m)$ was red in the original coloring. Therefore, after all right-to-left maxima have been changed to blue, the red entries still avoid $132$ and the blue entries still avoid $213$. Using the same argument as B\'ona~\cite{bona:a-new-upper-bou:}, it can be shown that the map from $1324$-avoiding permutations to the pairs of words $(e_\pi,v_\pi)$ (which have changed due to the new coloring) is still injective.

We now restrict this map to $1324$-avoiding involutions. Recall that in an involution, the inverse image of a right-to-left maximum is also a right-to-left maximum, and the inverse image of a left-to-right minimum is also a left-to-right minimum. Hence, given the pair of words $(e_\pi, v_\pi)$ for a permutation $\pi \in \Av^I(1324)$, the words $e_\pi$ and $v_\pi$ have the letter $\a$ in the same positions and have the letter $\d$ in the same positions. This is a significant restriction which yields a much smaller upper bound for $\ugr(\Av^I(1324))$ than the bound for $\gr(\Av(1324))$.

Let $h(x)$ be the generating function for pairs of words $(e_\pi, v_\pi)$ over the alphabet $\{\a,\b,\c,\d\}$ such that neither $e_\pi$ nor $v_\pi$ contain a $\c\b$ factor and such that $e_\pi$ and $v_\pi$ have all $\a$ entries in identical positions and all $\d$ entries in identical positions. It is a simple exercise in automata theory (for which we refer to Flajolet and Sedgewick~{\cite[I.4.2]{flajolet:analytic-combin:}}) to prove that
	\[h(x) = \f{1+x}{1-5x+x^2-x^3}.\]

%Let $h(x)$ be the generating function for pairs of words $(e_\pi, v_\pi)$ over the alphabet $\{\a,\b,\c,\d\}$ such that neither $e_\pi$ nor $v_\pi$ contain a $\c\b$ factor and such that $e_\pi$ and $v_\pi$ have all $\a$ entries in identical positions and all $\d$ entries in identical positions. Then, $h(x)$ satisfies the recurrence
%	\[h(x) = 2xh(x)\ds\sum_{k=0}^\infty (k+1)^2x^k + 1.\]
%This is seen by considering the last occurrence of an $\a$ or $\d$ in a given pair of words. Suppose this lies in position $i+1$. Then, the pair of words $([e_\pi]_i, [v_\pi]_i)$ --- where $[w]_i$ denotes the truncation of the word $w$ to the first $i$ letters --- also satisfies the given conditions. By assumption, there are two possibilities for the $(i+1)$st entry (either $\a$ or $\d$), yielding the $2x$ in the recurrence. Lastly, any letters after the $(i+1)$st position of each word can only be $\b$ or $\c$. Given that the words avoid a $\c\b$ factor, the remaining portion of each word must have the form $\b\cdots\b\c\cdots\c$. Given that the pair of words must be the same length, there are $(k+1)^2$ ways to complete the pair of words, where $k$ is the number of letters to be added to the end of each word. Finally, the added $1$ accounts for the pair of empty words in the recurrence.

%Noting that
%	\[\ds\sum_{k=0}^\infty (k+1)^2x^k = \f{1+x}{(1-x)^3},\]
%we solve the recurrence for $h(x)$ to find that
%	\[h(x) = \f{(1-x)^3}{1-5x+x^2-x^3}.\]
% solve(f = f*2*x*(1+x)/(1-x)^3+1,f);
% f := 1/(-1+5*x-x^2+x^3)*(-1+x)^3
Letting $r = \sqrt[3]{8+6\sqrt{78}}$, the reciprocal of the smallest positive root of the denominator of $h(x)$ is $3r/(14+r-r^2)$. Therefore,
	\[\ugr(\Av^I(1324)) \leq \f{3r}{14+r-r^2} < 4.84.\]

\section{Concluding Remarks}

\begin{table}
\begin{footnotesize}
$$
\begin{array}{ccccccccc}%{|c|c|c|c|c|c|c|c|c|}
\hline
&&&&&&&&\\[-8pt]
\bm{\beta}&\bm{2413}&\bm{2431}&\bm{3421}&\bm{1342}&\bm{2341}&\bm{4231}&\bm{1324}&\bm{1234}\\[1pt]\hline
&&&&&&&&\\[-9pt]
n=5&0&1&2&2&2&2&2&2
\\[1pt]\hline
&&&&&&&&\\[-9pt]
n=6&0&1&2&3&3&2&4&4
\\[1pt]\hline
&&&&&&&&\\[-9pt]
n=7&0&2&3&2&3&5&9&10
\\[1pt]\hline
&&&&&&&&\\[-9pt]
n=8&0&2&5&5&5&11&17&35
\\[1pt]\hline
&&&&&&&&\\[-9pt]
n=9&0&6&7&10&10&30&52&101
\\[1pt]\hline
&&&&&&&&\\[-9pt]
n=10&0&6&13&17&17&62&106&261
\\[1pt]\hline
&&&&&&&&\\[-9pt]
n=11&0&16&19&22&22&162&292&727
\\[1pt]\hline
&&&&&&&&\\[-9pt]
n=12&0&16&31&44&44&377&635&1865
\\[1pt]\hline
&&&&&&&&\\[-9pt]
n=13&0&45&51&68&68&973&1753&5127
\\[1pt]\hline
&&&&&&&&\\[-9pt]
n=14&0&45&82&127&127&2378&3954&13045
\\[1pt]\hline
&&&&&&&&\\[-9pt]
n=15&0&126&135&184&184&6116&10824&35735
\\[1pt]\hline
\end{array}
$$
\end{footnotesize}
\caption{The number of simple $\beta$-avoiding involutions of length $n$ for $n=5$, $\dots$, $15$, with columns sorted according to the $n=15$ data.}
\label{table-simple-enum}
\end{table}

In many ways this paper represents an initial foray into the topic of pattern-avoiding involutions, which has been considered very little in the past. It is natural to ask if the substitution decomposition might be used to enumerate any other sets of the form $\Av^I(\beta)$ for $|\beta|=4$. Table~\ref{table-simple-enum} shows that numbers of simple $\beta$-avoiding involutions of lengths $n=5$, $\dots$, $15$ (again computed with PermLab~\cite{PermLab1.0}). One may expect sets with fewer simple permutations to be easier to understand. Therefore this data suggests that it might be fruitful to apply our techniques to enumerate the $2431$- or $3421$-avoiding involutions. However, counting the $4231$- or $1324$-avoiding involutions using the substitution decomposition appears to be much less promising.

\bigskip
\noindent{\bf Acknowledgments:} We are very grateful to Michael Albert for adding support for involutions to his PermLab package~\cite{PermLab1.0}.

\bibliographystyle{acm}
\bibliography{../../refs}

\end{document}